\documentclass[oneside]{amsart}
\usepackage{amsthm,graphicx}
\usepackage{amssymb}
\usepackage[usenames]{color} %color
\usepackage{colortbl}
\usepackage[shortlabels]{enumitem}
\usepackage{caption}
\usepackage{subcaption}
\usepackage{tikz}
\usepackage{pgfplots}
\usepgfplotslibrary{fillbetween}
\usepackage[pagebackref=true,colorlinks=true]{hyperref}
\renewcommand*{\backref}[1]{}
\renewcommand*{\backrefalt}[4]{[{\tiny%
    \ifcase #1 Not cited.%
          \or Cited on page~#2.%
          \else Cited on pages #2.%
    \fi%
    }]}

\usepackage{chngcntr}% http://ctan.org/pkg/chngcntr
\counterwithin{equation}{section}
\theoremstyle{definition}
 \newtheorem{theorem}{Theorem}
 \newtheorem{proposition}[equation]{Proposition}
 \newtheorem{definition}[equation]{Definition}
 \newtheorem{remark}[equation]{Remark}
 \newtheorem{lemma}[equation]{Lemma}
 \newtheorem{corollary}[equation]{Corollary}
 \newtheorem{example}[equation]{Example}

\usepackage{etoolbox,xstring,xspace}
\makeatletter

\AtBeginDocument{%
\let\origref\ref
\renewcommand*\ref[1]{%
  \origref{#1}\xlabel{#1}}
}
\newrobustcmd*\xlabel[1]{%
   \ifcsdef{siteref@doc@#1}{}{\csgdef{siteref@doc@#1}{,}}%
    \@bsphack%
    \begingroup
       \csxdef{siteref@doc@#1}{\csuse{siteref@doc@#1},\thepage}%
         \protected@write\@auxout{}%
        {\string\SiteRef{siteref@#1}{\csuse{siteref@doc@#1}}}%
     \endgroup
     \@esphack%
}

\newrobustcmd*\SiteRef[2]{\csgdef{#1}{#2}}

\newrobustcmd*\xref[1]{%
\ifcsundef{siteref@#1}{%
     \@latex@warning@no@line{Label `#1' not defined}
     }{%
    \begingroup
      \StrGobbleLeft{\csuse{siteref@#1}}{2}[\@tempa]\relax%
      \def\@tempb{}%
      \@tempcnta=0\relax%
      \@tempcntb=\@ne\relax%
      \def\do##1{\advance\@tempcnta\@ne}%
      \expandafter\docsvlist\expandafter{\@tempa}%
       \def\do##1{%
         \ifnum\@tempcntb=\@tempcnta\relax%
            \hyperpage{##1}%
         \else
            \hyperpage{##1},%
          \fi%
          \advance\@tempcntb\@ne
       }%
       [\expandafter\docsvlist\expandafter{\@tempa}]\xspace%
    \endgroup
   }%
}

\makeatother
\newcommand{\add}[1]{}

\newcommand{\m}[1]{}

\newcommand{\used}[1]{(\text{Used\  on\  pages\  \xref{#1}\!\!){\ }}}
\newcommand{\laber}[1]{\label{#1} \used{#1}}

\newcommand\RR{\mathbb R}

\newcommand\ZZ{\mathbb Z}
\newcommand\NN{\mathbb N}
\newcommand\TT{\mathcal{T}}

\newcommand\G{\Gamma}

\newcommand{\supp}{\mathrm{supp}}

\newcommand{\vv}{\bf{verified}}
\newcommand{\edge}{\mathrm{edge}}
\newcommand{\vertex}{\mathrm{vertex}}
\newcommand{\node}{\mathrm{node}}
\newcommand{\dist}{\mathrm{dist}}

\newcommand{\h}{\text{\small h}}
\newcommand{\C}{\text{C}}

\reversemarginpar

\numberwithin{equation}{section}

\begin{document}

\date{\today}
\title{Sandpile solitons via smoothing of superharmonic functions}
\author[N. Kalinin]{Nikita Kalinin}\thanks{National Research University Higher School of Economics, Soyuza Pechatnikov str., 16, St. Petersburg, Russian Federation. Support from the Basic Research Program of the National Research University Higher School of Economics is gratefully acknowledged.} 
\author[M. Shkolnikov]{Mikhail Shkolnikov}\thanks{IST Austria. Klosterneuburg 3400, Am campus 1. Supported by ISTFELLOW program.}
%\thanks{
%Research is supported in part by the project   
%}

\address{National Research University Higher School of Economics, Soyuza Pechatnikov str., 16, St. Petersburg, Russian Federation} 

\email{nikaanspb\{at\}gmail.com}

\address{IST Austria, Klosterneuburg 3400, Am campus 1.} 

\email{mikhail.shkolnikov\{at\}gmail.com}

\keywords{Sandpile model, discrete harmonic functions, solitons, smoothing}
%, Steiner problem, tropical symplectic area
\begin{abstract}
Let $F:\ZZ^2\to \ZZ$ be the pointwise minimum of several linear functions. The theory of {\it smoothing} of integer-valued superharmonic function allows us to prove that under certain conditions there exists the pointwise minimal superharmonic function which coincides with $F$ ``at infinity''.

We develop such a theory to prove existence of so-called {\it solitons} (or strings) in a certain sandpile model, studied by S. Caracciolo, G. Paoletti, and A. Sportiello. Thus we made a step towards understanding the phenomena of the identity in the sandpile group for a square where  solitons appear according to experiments.  We prove that sandpile states, defined using our smoothing procedure, move changeless when we send waves (that is why we call them solitons), and can interact, forming {\it triads} and {\it nodes}. 
\end{abstract}
\maketitle

%\tableofcontents
\section{Introduction}

Periodic patterns in sandpiles (one of the simplest cellular automata) were studied by S. Caracciolo, G. Paoletti, and A. Sportiello in the pioneer work~\cite{firstsand}, see also Section 4.3 of~\cite{CPS} and Figure 3.1 in~\cite{book}, Figure 9a in~\cite{sadhu2011effect}. \m{history and motivation}Experimental evidence suggests that these patterns appear in many sandpile pictures and carry a number of remarkable properties: in particular, they are self-reproducing under the  action of waves. That is why we call these patterns {\it solitons} (Figure~\ref{fig_soliton}, left). 

The fact that the solitons appear as ``smoothings'' of piece-wise linear functions was predicted by T. Sadhu and D. Dhar in~\cite{sadhu2012pattern}. We provide a definition of the {\it smoothing} procedure (Definition~\ref{def_thetan}). We prove the existence (and uniqueness modulo translation) of solitons for all rational slopes and give certain estimates on their shape. We construct {\it triads} (Figure~\ref{fig_soliton}, right) --- three solitons meeting at a point --- and triads satisfy similar properties.  In addition, we accurately write the theory of sandpiles on infinite domains in the absence of references, though it is absolutely parallel to the finite case. This article contains the facts (Theorem~\ref{th_stabilfn}, Corollary~\ref{cor_wavegp}) that we need later to establish more general convergence in sandpiles, see~\cite{us}, from where we have extracted for better readability \cite{us_series} and this article.

\begin{center} 
\begin{figure}[h]
\includegraphics[width=0.2\textwidth]{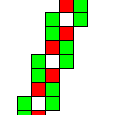}
\includegraphics[width=0.2\textwidth]{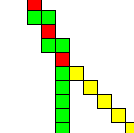}
\caption{These are local patterns for the soliton of direction $(1,3)$ and the triad made by solitons of directions $(0,-1),(1,-1),(1,2)$. White means three grains of sand, green -- two, yellow -- one, and red -- zero. }
\label{fig_soliton}
\end{figure}
\end{center} 
\m{more examples of solitones or nodes?}

\subsection{Sandpile patterns on $\ZZ^2$}
\m{definition of solitons}
A {\it state} of a sandpile is a function $\phi:\ZZ^2\to \ZZ_{\geq 0}$. We interpret $\phi(v)$ as the number of sand grains in $v\in\ZZ^2$. We can {\it topple} $v$ by sending four grains from $v$ to its neighbors, each neighbor gets one grain. If $\phi(v)\geq 4$, such a toppling is called {\it legal}. A {\it relaxation} is doing legal topplings while it is possible (for what this means for infinite graphs see Appendix~\ref{sec_locallyrelaxations}). A state $\phi$ is {\it stable} if $\phi\leq 3$ everywhere.

\begin{definition}
\label{def_waves}
Let $v\in \ZZ^2$ be such that $\phi(v)=\phi(w) = 3$ where $w$ is a neighbor of $v$. By {\it sending a wave} from $v$ we mean making a toppling at $v$, following by the relaxation. We denote the obtained state by $W_v\phi$.
\end{definition}

Note that after the first toppling the vertex $v$ has $-1$ grain, but $w$ now has $4$ grains, so it topples and $v$ has a non-negative number of grains again. 
We are interested in states which move changeless under the action of waves, such states were previously studied experimentally in~\cite{firstsand},\cite{CPS},\cite{book}.

\begin{definition}
A state $\phi$ on $\ZZ^2$ is called a {\it background} if there exists $v\in\ZZ^2$ such that $W_v\phi=\phi$. \m{it does matter which v we take. a state can contain any recurrent state encircled by 3}
\end{definition}
For example, such is the state $\phi\equiv 3$ decreased at any set of vertices with pairwise distances at least two.

\begin{definition}
Let $(p,q)\in\ZZ^2\setminus\{(0,0)\}$. A state $\phi$ is called $(p,q)$-{\it movable}, if there exists $v$ such that $W_v\phi(x,y) = \phi(x+p,y+q)$ for all $(x,y)$. A state $\phi$ is called $(p,q)$-{\it periodic} if there exist $p,q\in\ZZ$ such that $\phi(x,y)=\phi(x+p,y+q)$. A state $\phi$ is called  {\it line-shaped} if there exist constants $p,q,c_1,c_2$ such that the set $\{\phi\ne 3\}$ belongs to $\{(x,y)| c_1\leq px+qy\leq c_2\}$. 
\end{definition}

We classify all {\bf periodic line-shaped} movable states, we call them {\it solitons}. We also construct {\it triads}, i.e. three solitons meeting at a point, they are also movable. It seems to be more difficult to classify all movable states because many different backgrounds exist (see~\cite{book}, Chapter 5). We can vaguely guess the following: any $(p,q)$-movable state is equal to a soliton or triad on a background. So far we can prove the following theorem.

\begin{theorem}[See a proof in Section~\ref{proof_thmain}]
\label{th_main}
For each $p,q\in\ZZ, \mathrm{gcd}(p,q)=1$ there exists a unique (up to a translation in $\ZZ^2$) movable $(p,q)$-periodic line-shaped state. Furthermore, it is $(p',q')$-movable, where $p',q'\in\ZZ, p'q-pq'=1$.
\end{theorem}
Moreover, a movable $(p,q)$-periodic line-shaped state is always $(\frac{p}{\mathrm{gcd}(p,q)},\frac{q}{\mathrm{gcd}(p,q)})$-periodic.

\subsection{Superharmonic functions}

\begin{definition}
The {\it toppling function} of a relaxation is the function $\ZZ^2\to \ZZ_{\geq 0}$ counting the number of topplings at every point during this relaxation.
\end{definition}

It is known that the toppling function has bounded Laplacian and is minimal in a certain class of functions. Therefore, if we know the toppling function ``at infinity'', we can, in principle, reconstruct it. When we send $n$ waves towards a periodic movable line-shaped state, the toppling function is zero on one side of the set $\{\phi\ne 3\}$ and is equal to $n$ on another side. It is easy to guess (or find experimentally) that the toppling function in this case will be something like $F(x,y)=\min (px+qy, n)$ on one side of the set $\{\phi\ne 3\}$.  Hence we are looking for a point-wise minimal superharmonic integer-valued function which coincide with $F(x,y)$ at infinity, {\it a priopri} a pointwise minimum in such a class of functions can be $-\infty$ everywhere. We develop a theory of {\it smoothings} to prove that the pointwise minimum is reached by slicing from $F$ characteristic functions of certain sets, see Section~\ref{sec_propertiessmooth}. We also prove that a kind of monotonicity is preserved by doing this slicing, Section~\ref{sec_mono}. Using this we prove the existence of solitons. \m{some informal speculation explaining the main technical tools}

To study the interaction between solitons, when several of them meet at a point, and for the needs of~\cite{us}, we also study {\it triads} and {\it nodes}.
%, and now the technic of smoothing should be accompanied by several other technical tricks, see Section~\ref{sec_triads}.

\subsection{Sandpiles on infinite domains}
We could not find a satisfactory reference containing the theory of sandpiles on infinite domains (in particular, the Least Action Principle for waves). We hesitated about its inclusion here, because all the statements can be proven exactly in the same way as in the finite case. Finally we decided to present the theory of locally-finite relaxations in Appendix~\ref{sec_locallyrelaxations} where we define and study locally-finite relaxations.

\subsection{Acknowledgments}

We thank Andrea Sportiello for sharing his insights on perturbative
regimes of the Abelian sandpile model which was the starting
point of our work. We also thank Grigory Mikhalkin, who encouraged us to approach this problem. 

Also we thank Misha Khristoforov and Sergey
Lanzat who participated on the initial state of this project, when we
had nothing except the computer simulation and pictures. Ilia Zharkov,
Ilia Itenberg, Kristin Shaw, Max Karev, Lionel Levine, Ernesto Lupercio, Pavol \v Severa, Yulieth Prieto, Michael Polyak, Danila Cherkashin asked us a lot of questions and listened to us; not all of these questions found their answers here, but
we are going to treat them in subsequent papers.

%The research of the first author, Nikita Kalinin, (). %is funded by the SNSF PostDoc.Mobility grant 168647. 
%
%The second author, Mikhail Shkolnikov, 
%(University of Geneva, Switzerland) is supported in part by the grant 159240 of the Swiss National Science
%Foundation as well as by the National Center of Competence in Research
%SwissMAP of the Swiss National Science Foundation. 

\section{Smoothing, its relation to waves}

The discrete Laplacian $\Delta$ of a function $F:\ZZ^2\to\RR$ is defined as $$\Delta F(x,y) = -4F(x,y)+F(x+1,y)+F(x-1,y)+F(x,y+1)+F(x,y-1).$$ A function $F$ is called {\it harmonic} (resp., {\it superharmonic}) on $A\subset \ZZ^2$ if $\Delta F=0$ (resp., $\Delta F\leq 0$) at every point in $A$.
\begin{lemma}
\laber{lem_minharmonic}
If $F,G$ are two superharmonic functions on $A\subset\ZZ^2$, then $\min(F,G)$ is a superharmonic function on $A$.
\end{lemma}

\begin{proof}
Let $v\in A$. Without loss of generality, $F(v)\leq G(v)$. Then, $\Delta \min(F,G)(v)\leq \Delta F(v)\leq 0.$
\end{proof}

\begin{definition} \m{notation}
The {\it deviation set} $D(F)$ of a function $F$ is the set of points where $F$ is not harmonic, i.e. 
$$D(F) =\{(x,y)\in\ZZ^2| \Delta F(x,y)\ne 0\}.$$
\end{definition}

For $A\subset \ZZ^2,\C>0$, we denote by $B_\C(A)\subset \ZZ^2$ the set of points whose distance to $A$ is at most $\C$.

\begin{definition}
\laber{def_thetan}
For $n\in\mathbb{N}$ and a superharmonic function $F:\ZZ^2\to\ZZ$ we define 
$$\Theta_n(F)=\{G:\ZZ^2\to \ZZ| \Delta G \leq 0, F-n\leq G\leq F, \exists \C>0, \{F\ne G\}\subset B_{\C}(D(F))\}.$$

In plain words, $\Theta_n(F)$ is the set of all integer-valued superharmonic functions $G\leq F$, coinciding with $F$ outside a finite neighborhood of $D(F)$, whose difference with $F$ is at most $n$. 
 Define $S_n(F):\ZZ^2\to\ZZ$ to be $$S_n(F)(v)=\min\{G(v)|{G\in\Theta_n(F)}\}.$$ We call $S_n(F)$ {\it the $n$-smoothing of $F$}. Note that $S_n(F)\geq F-n$. %Denote $\Theta_F=\bigcup\Theta_n(F)$.\add{do 
\end{definition}
\m{smoothing}

\begin{lemma}
\label{lem_notevident} If $F\leq F'$ then $S_n(F)\leq S_n(F')$ for each $n\in\ZZ_{\geq 0}$. \m{not quite evident} 
\end{lemma}
\begin{proof}
Suppose that there exist $v\in\ZZ^2$ and $G'\in \Theta_n(F')$ such that $G'(v)< S_n(F)(v)$. The set $\{G'\ne F'\}$ belongs to a certain $\C$-neighborhood of $D(F')$. Consider the function $G=\min(S_n(F),G')$. Clearly, $\Delta G\leq 0$, $F-n\leq %F'-n\leq 
G\leq F$. We will prove that the set $\{S_n(F)> G'\}$ belongs to a $\C+n+1$ neighborhood of $D(F)$ and thus $G\in \Theta_n(F)$ and $G(v)<S_n(F)(v)$ which is a contradiction.

Suppose that $G'(v)< S_n(F)(v)$ and the distance between $v$ and $D(F)$ is at least $\C+n+1$. There is a path $v=v_0,v_1,\dots, v_k$ of length at most $\C$ from $v$ to some $v_k\in D(F')$ and $F$ is harmonic near this path. Then, at $v'$ we have $\Delta (F'-F)<0$ and, for some $i\leq k$, we have that $n\geq (F'-F)(v_0)=(F'-F)(v_1)=\dots =(F'-F)(v_i)>(F'-F)(v')$ where $v'$ is a neighbor of $v_i$. Therefore $v'$ has a neighbor $v'_1$ such that $(F'-F)(v')>(F'-F)(v'_1)$, etc, therefore we can construct a path $v'_1,v'_2,\dots$ of length at least $n+1$ in $\ZZ^2\setminus D(F)$, and we have that  $(F'-F)(v'_{n+1})< 0$ which is a contradiction.
\end{proof}

Let us fix $p_1,p_2,q_1,q_2,c_1,c_2 \in \ZZ$ such that $p_1q_2-p_2q_1=1$. Consider the following functions on $\ZZ^2$:  
\begin{equation}\label{eq_edge}
\Psi_{\edge}(x,y)=\min(0,p_1x+q_1y),
\end{equation}
\begin{equation}\label{eq_tripode}
\Psi_{\vertex}(x,y)=\min(0,p_1x+q_1y,p_2x+q_2y+c_1),
\end{equation}
\begin{equation}\label{eq_node}
\Psi_{\node}(x,y)=\min\Big(0, p_1x+q_1y,p_2x+q_2y+c_1, (p_1+p_2)x+(q_1+q_2)y+c_2 \Big).
\end{equation}

Our main technical result is the following theorem.

\begin{theorem}\label{th_stabilfn}
Let $F$ be a) $\Psi_{\edge}$, b)$\Psi_{\vertex}$, or c)$\Psi_{\node}$. The sequence of $n$-smoothings $S_n(F)$ of $F$ stabilizes eventually as $n\to\infty$, i.e. there exists $N>0$ such that $S_n(F)\equiv S_N(F)$ for all $n>N$. 
\end{theorem}

See a proof of (a) in Section~\ref{sec_proofedge} and a proof of (b,c) in Section~\ref{proof_vertices}. The problems to overcome in proofs are as follows: the deviation set $D(F)$ is infinite, and we need to prove that it ``flows'' only locally and can not significantly spread when we increase $n$. Then, even if the flow of $D(S_n(F))$ is restrained locally when $n$ increases, the deviation locus, in principle, can encircle growing regions where $S_n(F)$ is harmonic almost everywhere. After taming these and other technicalities, the proof amounts to the fact that there exists no integer-valued linear function which is less than $F$ an a non-empty compact set.

\begin{definition}
The pointwise minimal function in $\bigcup\Theta_n(F)$, which exists by
Theorem~\ref{th_stabilfn}, is called {\it the canonical smoothing of $F$} and is denoted by $\theta_F$. \m{relation to sandpiles}
\end{definition}
\begin{remark}
\label{rem_theta}Note that $\Delta \theta_F\geq -3$ because otherwise we could decrease $\theta_F$ at a point violating this condition, preserving superharmonicity of $\theta_F$, and this would contradict to the minimality of $\theta_F$ in $\bigcup\Theta_n(F)$.
\end{remark}

We think of $\ZZ^2$ as the vertices of the graph whose edges connect points with distance one. 
Let $F$ be $\Psi_{\edge}$, $\Psi_{\vertex}$, or $\Psi_{\node}$,  we write $$F(x,y)=\min_{(i,j)\in A}(ix+jy+a_{ij}).$$
Consider a sandpile state $\phi=3 + \Delta \theta_F$. By Remark~\ref{rem_theta}, $\phi\geq 0$ and $\phi$ is a stable state because $\theta_F$ is superharmonic. Let $v\in \ZZ^2$ be a point far from $D(\theta_F)$. Let $F$ be equal to $i_0x+j_0y+a_{i_0j_0}$ near $v$. The following corollary say, informally, that sending a wave from $v$ increases the coefficient $a_{i_0j_0}$ by one. 
\begin{corollary}
\label{cor_wavegp}
In the above conditions, $W_v\phi=3+\Delta\theta_{F'}$ where $W_v$ is the sending wave from $v$ (Definition~\ref{def_waves}) and $$F'(x,y)=\min\big(i_0x+j_0y+a_{i_0j_0}+1,\min_{(i,j)\in A, (i,j)\ne (i_0,j_0)}(ix+jy+a_{ij})\big).$$ 
\end{corollary}
\begin{proof} Let $H^v_\phi$ \eqref{eq_topplewave} be the toppling function of the wave from $v$. Since $$W_v\phi = \phi +\Delta H^v_\phi =3+\Delta(\theta_{F}+H^v_\phi),$$ we want to prove that  $H^v_\phi = \theta_{F'}-\theta_F$. It follows from Lemma~\ref{lem_notevident} that $\theta_{F'}-\theta_F\geq 0$.
By the Least Action Principle for waves (Proposition~\ref{prop_waveleast}) we have that $\theta_{F'}-\theta_F\geq H^v_\phi$ because $\theta_{F'}-\theta_F=1$ at $v$, $\theta_F'-\theta_F\geq 0$ and $\phi + \Delta(\theta_F'-\theta_F) = 3+\theta_{F'}$ is a stable state. On the other hand, the function $\theta_F+ H^v_\phi$ coincides with $\theta_{F'}$ outside of a finite neighborhood of $D(F')$ and is superharmonic. Therefore, \m{connection to wave action}by the definition of $\theta_{F'}$,  we see that $\theta_{F'}\leq \theta_F+ H^v_\phi$ and this finishes the proof.
\end{proof}

\begin{remark}As we will see later, all sandpile solitons are of the form $3+\Delta\theta_{\Psi_\edge}$.
\end{remark}

\section{Holeless functions}

\m{technical about well-definedness}
\begin{definition}
\label{def_holeless}
We say that a function $F:\ZZ^2\to \ZZ$ is {\it holeless} if there exists $\C>0$ such that $B_\C(D(F))$ contains all the connected components of $\ZZ^2\setminus D(F)$ which belong to some finite neighborhood of $D(F)$. When we want to specify the constant $\C$ we write that $F$ is $\C$-holeless.
\end{definition}

\begin{example}
\label{ex_small}
The functions $F=\Psi_{\edge},\Psi_{\vertex},\Psi_{\node}$ (see \eqref{eq_edge},\eqref{eq_edge},\eqref{eq_node}) are holeless just because $\ZZ^2\setminus D(F)$ has no components which belong to a finite neighborhood of $D(F)$.
\end{example}

\begin{lemma}
\laber{lem_finiteneigh}
If $F$ is $\C$-holeless, then for each $G\in \Theta_n(F)$ the set $\{F\ne G\}$ is contained in $B_{\max(n,\C)}(D(F))$.
\end{lemma}
\begin{proof}
Let $A_n=\{v\in\ZZ^2| G(v)=F(v)-n\}$. If $v\in A_n\setminus D(F)$ then from the superharmonicity of $G$ and harmonicity of $F$ at $v$ we deduce that  all neighbors of $v$ belong to $A_n$. Therefore the connected component of $v\in A_n$ in $\ZZ^2\setminus D(F)$ belongs to $A_n$, which, in turn, belongs to a finite neighborhood of $D(F)$ because there belongs the set $\{F\ne G\}$. Thus $A_n$ belongs to $\C$-neighborhood of $D(F)$.  %Therefore $v$ is contained in $B_\C(D(F))$. 
By the same arguments, each point of $A_{n-1} = \{G=F-n+1\}$ is contained in the 1-neighborhood of $D(F)\cap A_n$ or, together with its connected component of $\ZZ^2\setminus D(F)$ belongs to $A_{n-1}$, i.e. is contained in $B_\C(D(F))$, $A_{n-2}$ is contained in the $2$-neighborhood of $D(F)\cap A_n$ or in $1$-neighborhood of $D(F)\cap A_{n-1}$, or in $B_\C(D(F))$, etc. 
\end{proof}

\begin{corollary}%[\citepage{proof_finiteness}]
\laber{cor_finiteness}
If $F$ is $\C$-holeless for some $\C>0$, then for each $n\geq 0$ the function $S_n(F)$ belongs to $\Theta_n(F)$.
\end{corollary}

%\begin{proof}%[Proof of Proposition~\ref{prop_finiteness}]
%\label{proof_finiteness}
%By Lemma~\ref{lem_minharmonic} the function $S_n(F)$ is superharmonic and by Lemma~\ref{lem_finiteneigh} 
%$$\{F\ne S_n(F)\}\subset B_{ \max(n,\C)}(D(F)).$$
%\end{proof}

\begin{corollary}%[\citepage{proof_grows}]
\laber{cor_grows}
Let $F$ be one of $\Psi_{\edge},\Psi_{\vertex},\Psi_{\node}$ (see \eqref{eq_edge},\eqref{eq_tripode},\eqref{eq_node}). 
Then for each $n\geq 1$ we have $$\dist\Big(D(F),\big\{F\ne S_n(F)\big\}\Big)\leq n.$$
\end{corollary}

%\begin{proof}%[Proof of Lemma~\ref{lemma_grows}]
%\label{proof_grows}
%It follows from Lemma~\ref{lem_finiteneigh} and Example~\ref{ex_small}.
%\end{proof}

\section{Smoothing by steps} 
\label{sec_propertiessmooth}

Let $F,G$ be two superharmonic integer-valued functions on $\ZZ^2$. Suppose that $H=F-G$ is non-negative and bounded. Let $m$ be the maximal value of $H$. Define the functions $H_k, k=0,1,\dots, m$ as follows:
\begin{equation}
\label{eq_chi}
H_k(v) =\chi(H\geq k) = \left\{
\begin{aligned}
1,\  & \mathrm{if}\ H(v)\geq k,\\
0,\ & \mathrm{otherwise}.
\end{aligned}
\right.
\end{equation}

\begin{lemma}
\label{lemma_slice}
In the above settings, the function $F-H_m$ is superharmonic.
\end{lemma}
\begin{proof}
Indeed, $F-H_m$ is superharmonic outside of the set  $\{x | H(x)=m\}$. Look at any point $v$ such that $H(v)=m$. Then we conclude by $$4(F-H_m)(v)= 4G(v)+4(m-1)\geq \sum_{w\sim v} G (w)+4(m-1)\geq \sum_{w\sim v} (F-H_m)(w).$$
\end{proof}

%\begin{figure}[h]
%\includegraphics[width=8cm]{smoothing.pdf}
%\label{fig_smoothing}
%\end{figure}

We repeat the procedure in Lemma~\ref{lemma_slice} for $F-H_m$; namely, consider $F-H_m-H_{m-1}, F-H_m-H_{m-1}-H_{m-2}$, etc. We have $$H=H_m+H_{m-1}+ H_{m-2}+\dots+ H_1,$$ and it follows from subsequent applications of Lemma~\ref{lemma_slice} that all the functions $F-\sum_{n=m}^{m-k+1} H_n$ are superharmonic, for $k=1,2,\dots,m$. Also, it is clear that $$0 \leq \left(F-\sum_{n=m}^{m-k+1} H_n\right)-\left(F-\sum_{n=m}^{m-k} H_n\right) = H_{m-k}\leq 1$$ at all $v\in\G,k=0,\dots,m$.

\begin{remark}
It follows from definition of $H_i$ that $\supp(H_m)\subset \supp(H_{m-1})\subset\dots$ and, hence, $|H_i-H_{i+1}|\leq 1$ for $i=1,\dots, m-1$. \end{remark}

Consider a superharmonic function $F.$ We are going to prove that two consequitive smoothings (see Definition~\ref{def_thetan}) of $F$ differ at most by one at every point of $\ZZ^2$.

\begin{proposition}
\laber{prop_slicingFn}For all $n\in\NN$
$$|S_n(F)-S_{n+1}(F)|\leq 1.$$
\end{proposition}
\begin{proof}
By definition, \m{standard simple lemma. we just need to subtract the maximum of difference} $S_n(F)\geq S_{n+1}(F)$ at every point of $\mathbb{Z}^2$. If the inequality $|S_n(F)-S_{n+1}(F)|\leq 1$ doesn't hold, then the maximum $M$ of the function $H=S_{n}(F)-S_{n+1}(F)$ is at least $2$. We will prove that 
$$S_n(F)-\chi (H\geq M)\geq F-n.$$

Namely, by Lemma \ref{lemma_slice} the function $S_n(F)-\chi (H\geq M)$ is superharmonic. Suppose that $$S_n(F)-\chi (H\geq M)< F-n \text{\ at\ a\ point\ $v\in\ZZ^2$}.$$ 
Since the support of $\chi(H\geq 1)$ contains the support of $\chi(H\geq M)$, we have that $$S_n(F)-\chi (H\geq M) - \chi (H\geq 1) <F-(n+1) \text{\ at\  $v$}.$$ But this contradicts to $$S_n(F)-\chi(H\geq 1)-\chi (H\geq M)\geq S_{n+1}(F)\geq F-(n+1).$$

Therefore $S_n(F)-\chi (H\geq M)\in\Theta_n(F)$ which contradicts the minimality of $S_n(F)$.

\end{proof}
\m{that were all the faction about slicing we need}

\begin{corollary}\laber{cor_fnmin}
Proposition~\ref{prop_slicingFn} and Corollary~\ref{cor_grows} imply that the function $S_{n+1}(F)$ can be characterized as the point-wise minimum of all superharmonic functions $G$ such that $S_n(F)-1\leq G\leq S_n(F)$  and $S_n(F)-G$ vanishes outside some finite neighborhood of $D(S_n(F))$ (recall that the distance between $D(F), D(S_n(F))$ is at most $n$). % (Definition~\ref{def_thetan}). 
In other words, $n$-smoothing $S_n(F)$ of $F$ is the same as $1$-smoothing of $(n-1)$-smoothing $S_{n-1}(F)$ of $F$.
\end{corollary}

\begin{corollary}
\laber{cor_1smoothing}
If $S_n(F)\ne S_{n+1}(F)$ then there exists a point $v_0$ such that $S_{n+1}(F)(v_0)=F(v_0)-(n+1)$.
\end{corollary}
Indeed, if there is no such a point, \m{in fact, very important corollary} then $S_{n+1}(F)\geq F-n$ and therefore $S_{n+1}(F)=S_n(F)$.

%\begin{corollary}
%If $F\leq G$ then $S_1(F)\leq S_1(G)$. \add{proof?}
%\end{corollary}

\begin{remark}
\laber{rem_cases}
Let $F(x,y)=\min(x,y,0)$ or $F(x,y)=\min(x,y,x+y,c)$ for $c\in\ZZ_{\geq 0}$. Then it is easy to check that $S_1(F)(x,y)=F(x,y)$ and therefore $\Theta_F=\{F\}$ (see Definition~\ref{def_thetan} for the notation).  %In a sense, this observation provides a stronger version of Lemma~\ref{lemma_linearmonoton}.
\end{remark}
\m{this is an important example!}
\section{Monotonicity while smoothing}
\label{sec_mono}
\begin{definition}
\laber{def_monotone}
Let $e\in\ZZ^2\setminus\{(0,0)\}$. We say that a function $F:\ZZ^2\to\ZZ$ is {\it $e$-increasing} if 
\begin{enumerate}[a)]
\item $F$ is a smoothing of a \m{main tool, probably with bugs} holeless function,
\item $F(v)\leq F(v+e)$ holds for each $v\in\mathbb{Z}^2$,
\item there exists a constant $\C>0$ \m{convexity condition} such for each $v$ with $F(v)=F(v-e)$, the first vertex $v-ke$ in the sequence $v,v-e,v-2e,\dots$, satisfying $F(v-ke)<F(v-(k-1)e)$, belongs to $B_\C(D(F))$.
\end{enumerate}
\end{definition}
%Recall that $D(G)=\{w\in\ZZ^2|\Delta G(w)\ne 0\}$.

\begin{example}
\laber{ex_monotone}
Let $F(x,y)=\min(px+qy,0)$ where $p,q\in\ZZ$. Note that $F$ is $(e_1,e_2)$-increasing if and only if $pe_1+qe_2\geq 0$. In particular, $F$ is $(0,1)$-increasing if $q>0$ and $(e_1,e_2)$-increasing if $p<0<q$ and $0\leq e_1\leq q-1,e_2\geq|p|$.
\end{example}

\begin{lemma}
|f $F$ is $e$-increasing, then $S_1(F),$ the $1$-smoothing of $F$, is also $e$-increasing. \m{main technical tool, concavity of functions}
\end{lemma}

\begin{proof} 

%\begin{figure}[h]
%\includegraphics[width=8cm]{monotone.pdf}
%\caption{monotone function, z increasing in the new termnology}
%\label{fig_monotone}
%\end{figure}
Corollary~\ref{cor_fnmin} gives the property a) of Definition~\ref{def_monotone} for free.
To prove that $S_1(F)$ satisfies b) in Definition~\ref{def_monotone} we
 argue {\it a contrario}. Let $H= F-S_1(F)$. Suppose that the set $$A=\{v\in\mathbb{Z}^2 | F(v-e)=F(v),H(v-e)=0, H(v)=1\}$$ is not empty. Since $H|_A=1$, $A\subset B_\C(D(G))$. Consider the set $$B=\{v| H(v)=0, \exists n\in\mathbb{Z}_{> 0}, v+n\cdot e\in A, F(v)=F(v+n\cdot e)\}.$$ 
 
Consider $v\in B$. Since $v+n\cdot e \in A\subset B_C(D(F))$ and $F(v)=F(v+n\cdot e)$, then c) in Definition~\ref{def_monotone} impose an absolute bound on $n$  and therefore $B$ belongs to a finite neighborhood of $D(F)$.  Consider the following function $$\tilde F=S_1(F)-\sum_{v\in B}\delta_v.$$ 
It is easy to verify that $\tilde F(v)\leq \tilde F(v+e)$ for each $v$. Note that $\Delta \tilde F(v)\leq \Delta S_1(F)(v)\leq 0$ automatically for all $v\in\ZZ^2\setminus B$. 
 Take $v\in B$. Since $v+n\cdot e\in A$ for some $n\in\ZZ_{>0}$, we have
$$4\tilde F(v)=4S_1(F)(v+n\cdot e)\geq \sum_{w\sim v}S_1(F)(w+n\cdot e)\geq \sum_{w\sim v}\tilde F(w).$$ 

 Therefore $\tilde F$ is superharmonic, and satisfies $F\geq \tilde F\geq F-1$ by construction, which contradicts to the minimality of $S_1(F)$ in $\Theta_1(F)$. %See Figure~\ref{fig_monotone} for an illustration of the above arguments. 
 
  Finally, by Corollary~\ref{cor_finiteness} the sets $\{F-S_1(F)\}$ and $D(S_1(F))$ belong to $B_\C(D(F))$ for some $\C>0$. Therefore the fact that $|S_1(F)-F|\leq 1$ (Proposition~\ref{prop_slicingFn}) gives c) with the constant $\C+|e|+1$. 
\end{proof}

\begin{corollary}\laber{cor_monotf}
Let $F$ be one of $\Psi_{\edge},\Psi_{\vertex},\Psi_{\node}$ (Eqs.~\eqref{eq_edge},\eqref{eq_tripode}, \eqref{eq_node}). Let $e\in\ZZ^2\setminus \{(0,0)\}$. If $F$ is $e$-increasing, then $S_n(F)$ is also $e$-increasing.
\end{corollary}

\m{concrete example}

The following remark follows from the definition of smoothing.
\begin{remark}
\laber{rem_linear}
Let $F:\ZZ^2\to\ZZ, p,q,r\in\ZZ$. Let $G(x,y)=F(x,y)-px-qy-r$. Then $S_n(F)(x,y)-(px+qy+r) = S_n(G)(x,y)$.
\end{remark}
\m{change of coordinate}

%\begin{lemma}
%\laber{lemma_linearmonoton}Suppose that the area of a lattice triangle with vertices $(0,0),(p,q),(p'q')\in\ZZ^2$ is $1/2$. Let $c,c',d,d'\in\ZZ$ be constants.
%Consider the functions $$H_1(x,y)=\min(px+qy+c,p'x+q'y+c',d),$$ 
%$$H_2(x,y)=\min(px+qy+c,p'x+q'y+c',d,(p+p')x+(q+q')y+d').$$ 
%
%
%Then, there exists no linear function $H'(x,y)=p''x+q''y+a''$ with $$(p'',q'')\in\ZZ^2\setminus\{ (0,0),(p,q),(p',q')\}$$ such that the set $\{(x,y)|H'(x,y)<H_1(x,y)\}$ is bounded,\m{key lemma about smoothing. Not quite related to where it is used} and there exists no linear function $H'(x,y)=p''x+q''y+a''$ with $$(p'',q'')\in\ZZ^2\setminus\{(0,0),(p,q),(p',q'),(p+p',q+q')\}$$ such that the set $\{(x,y)|H'(x,y)<H_2(x,y)\}$ is bounded.
%\end{lemma}
%
%\begin{proof} 
%Applying $\operatorname{SL}(2,\ZZ)$-change of the coordinates and a parallel translation, we may restrict ourselves to the model case $$H_1(x,y)=\min(x,y,0)\text{\ and\ }H'(x,y)=Px+Qy+R$$ for some $P,Q\in\ZZ, R\in\RR$. Consider the restriction of $H'$ to the ray $(t,0)_{t\in \RR_{\geq 0}}$. Since $H'$ must be bigger than $H$ far from zero, we see that $P\geq 0$.  By considering $\text{ rays\ } \{(0,t)\}_{t\in \RR_{\geq 0}} \text{and\ } \{(-t,-t)\}_{t\in \RR_{\geq 0}},$ we conclude that $Q\geq 0$ and $P+Q\leq 1$. Using the fact that $(p'',q'')\in\ZZ^2$ we arrive to a contradiction.
%The second part of the statement can be proven similarly by reduction to the model case $H_2=\min(x,y,x+y,0)$.
%\end{proof}
%

\begin{lemma}
\laber{lem_alldirection}
Let $F$ be $\Psi_\vertex$ or $\Psi_\node$. Then there exists $k>0$ that for each $n>0$ for each square $S$ of size $k\times k$ inside the set $\{F\ne S_n(F)\}$ the function $S_n(F)$ is not linear on $S$.
\end{lemma}
\begin{proof}\m{very crucial lemma}
 We associate each linear function $px+qy+c$ in $F$ with the point $(p,q)\in\ZZ^2$. Then, a triangle of area $1/2$ is associated to $\Psi_\vertex$ and a parallelogram of area $1$ is associated to $\Psi_\node$. Pick any vertex of such a polygon (triangle or parallelogram) $\Delta$. Using Remark~\ref{rem_linear} we may suppose that this point is $(0,0)$. Then, such a function $F$ is $e$-monotone for all $e$ in the dual cone for the cone at $(0,0)$ in $\Delta$ (the dual cone is the set of vectors which have a non-negative scalar product with vectors from a given cone). By Corollary~\ref{cor_monotf}, $S_n(F)$ is also monotone in this direction. Therefore, if $S_n(F)$ is $m_1x+m_2y+m_3$ on $S$, the point $(m_1,m_2)$ must belong to the cone at $(0,0)$ in $\Delta$. By doing that for each vertex of $\Delta$ we obtain that $(m_1,m_2)$ belongs to $\Delta$. To be able to do that we need to assume that $k$ is big enough, namely, bigger than twice the length of each primitive vector from the edges of the dual cones considered above (this is a finite set of vectors).

Then, the fact that $(m_1,m_2)\in \Delta$ (and hence $(m_1,m_2)$ is a vertex of $\Delta$) contradicts the superharmonicity of $S_n(F)$. Namely, by Remark~\ref{rem_linear} we may assume that $(m_1,m_2)=(0,0)$ and $F$ is given as in \eqref{eq_tripode} or \eqref{eq_node}. Therefore $S_n(F)$ is a constant $m_3<0$ on $S$. Consider the set $\{S_n(F) = m_3\}$. There is a direction $e$ such that $F$ is $e$-increasing and for each $v\in \{S_n(F) = m_3\}$ we have that $v+le$ belongs to $\{S_n(F)=0\}$ for some $l\in \ZZ_{>0}$. Then we go from the center of the set $\{S_n(F) = m_3\}$ in the direction $e$, and we find a vertex $v\in \{S_n(F) = m_3\}$ such that $S_n(F)\geq m_3$ at all neighbors of $v$ and $S_n(F)> m_3$ at one of the neighbors, and this contradicts to the superharmonicity of $S_n(F)$. 
%Denote by $V$ the set of primitive integer vectors $v$ going in the directions of the edges of the tropical curve corresponding to $F$, in the direction ``from the vertex'' of this tropical curve. Take $k$ bigger than twice the maximal length of the vectors in $V$.
%
%Suppose that $S_n(F)$ is equal to a linear function $L$ on such a square $S$.  By Lemma~\ref{lemma_linearmonoton}, $\{L\leq F\}$ is unbounded, and so contains an infinite part of a ray in a direction $v\in V$.
%Using Remark~\ref{rem_linear} we can add a linear function \add{not good explanation like a ewdge} to $F$, so we may suppose that $F$ is $v$-increasing, furthemore, for each $z\in \ZZ^2$, $F(z+k\cdot v)=F(z+(k+1)\cdot v)$ for all $k$ big enough. \add{better e instread if v}
%
%By the choice of $v$, we have that $S_n(F)$ is not $v$-increasing on $S$, but $F$ is $v$-increasing, thus this contradicts to Lemma~\ref{lemma_linearmonoton}.
\end{proof}

For $\Psi_\edge$ the similar result holds, see Lemma~\ref{lemma_edge} (perhaps, with a more direct proof).

\section{Discrete superharmonic integer-valued functions.}

%\add{pu to appendix}

%\subsection{Discrete }
\m{this is a section for auxilirly statements}

Throughout the paper we denote all absolute constants by $\C$, when we want to stress that $\C$ depends on other constants such as $k,\dots$ we write it as $\C({k,\dots})$ correspondingly. We will also omit writing ``there is an absolute constant $\C$ with the following property...''.

\begin{lemma}(\cite{duffin}, Theorem 5)
\laber{lemma_duffin} 
Let $R>1,v\in\ZZ^2$, and $F:B_R(v)\cap \ZZ^2\to\RR$ be a discrete {\bf non-negative} harmonic function. Let $v'\sim v$, then $$|F(v')-F(v)|\leq \frac{\C\cdot \max_{v\in B_R(v)} F(v)}{R}.$$%\add{probably always write a ball with R not r?}
\end{lemma}
%Note that in~\cite{duffin} this result is formulated for a discrete harmonic function $F$ on $\ZZ^3$ lattice, but  the two-dimensional case follows if we substitute $F(x,y,z)=F(x,y)$ \add{find another reference?} for all $z\in \ZZ$. \m{definition of derivatives, explanation of the lemma}

Morally, this lemma provides an estimate on a derivative of a discrete harmonic function. We call $\partial_xF(x,y)=F(x+1,y)-F(x,y)$ the {\it discrete derivative} of $F$ in the $x$-direction. The derivative $\partial_y$ in the $y$-direction is defined in a similar way. We denote by $\partial_\bullet F$ the discrete derivative of a function $F$ in any of directions $x$ or $y$.

\begin{lemma}[Integer-valued discrete harmonic functions of sublinear growth]
\laber{lemma_harmonic} %\add{write on the paper all the results wtha we use from other people and how prove them the idea.}
Let $v\in\ZZ^2$ and $\mu>0$ be a constant. Let $R>4\mu \C$. For a discrete {\bf integer-valued} harmonic function $F:B_{3R}(v)\cap\ZZ^2\to\ZZ$, the condition $|F(v')|\leq \mu R$ for all $v'\in B_{3R}(v)$ implies that $F$ is linear in $B_{R}(v)\cap \ZZ^2$. 
\end{lemma}
\begin{proof} Consider $F$ which satisfies the hypothesis of the lemma.
Note that $0\leq F(v')+\mu R\leq 2\mu R$ for $v'\in B_{3R}(v)$ and applying Lemma~\ref{lemma_duffin} for $B_{2R}(v)$ yields $$|\partial_\bullet F(v')|\leq \frac{\C\cdot 2\mu R}{R}=2\mu \C \text{,\ for all $v'\in B_{2R}(v)$}.$$ Then, applying it again for $0\leq \partial_\bullet F(v')+2\mu C\leq 4\mu C$ yields
\begin{equation*}
|\partial_{\bullet}\partial_{\bullet}F(v')|\leq \frac{4\mu \C}{R}<1 \text{,\ for\  $v'\in B_{R}(v)$\ if\ $R>4\mu \C$.}
\end{equation*}  Since $F$ is integer-valued, all the derivatives $\partial_{\bullet}\partial_{\bullet}F$ are also integer-valued. Therefore all the second derivatives of $u$ are identically zero in $B_{R}(v)$, which implies that $F$ is linear in $B_{R}(v)$.
\end{proof}

 Let $A$ be a finite subset of $\ZZ^2$, $\partial A$ be the set of points in $A$ which have neighbors in $\ZZ^2\setminus A$. Let $F$ be any function $A\to \ZZ$.
 
\begin{lemma}
\laber{lemma_nabla} In the above hypothesis the following equality holds: 
$$\sum_{v\in A\setminus\partial A} \Delta F(v) = \sum_{\substack{v\in \partial A,\\ v'\in A\setminus\partial A, v\sim v'}} \big(F(v)-F(v')\big).$$
\end{lemma}
\begin{proof}
We develop left side by definition of $\Delta F$. All the terms $F(v)$, except for the vertices $v$ near $\partial A$, cancel each other. So we conclude by a direct computation.
\end{proof}

\begin{definition}
\laber{def_poisson}
For $v\in\ZZ^2$ we denote by $G_v:\ZZ^2\to\RR$ the function with the following properties: 
\begin{itemize}
\item $\Delta G_v(v)=1$,
\item $\Delta G_v(w)=0$ if $w\ne v$,
\item $G_v(v)=0$,
\item $G_v(w)= \frac{1}{2\pi}\log|w-v|+\C+O\left(\frac{1}{|w-v|^2}\right)$ when $|w-v|\to\infty$, where $c$ is some constant.
\end{itemize}
It is a classical fact that $G_v$ does exist and is unique (\cite{MR0040555}, (15.12), or~\cite{MR2677157}, p.104, see~\cite{MR1415236}, Remark 2, for more terms in the Taylor expansion).
\end{definition}
\begin{corollary}
\label{cor_poisson}
Let $v=(0,0)$. By a direct calculation we conclude that $$|\partial_\bullet\partial_\bullet G_v(x,y)|\leq \frac{\C}{(x^2+y^2+1)}.$$ 
\end{corollary}

\begin{lemma}
\label{lemma_log}
The following inequality holds for all $N\in\ZZ_{>0}, v\in\ZZ^2$:\m{lemma to cure any function to harmonic}
$$\sum\limits_{-N\leq x,y\leq N} |\partial_\bullet\partial_\bullet G_{v}(x,y)|\leq \C\ln N.$$
\end{lemma}
\begin{proof}
The maximum of this sum is attained when $v=(0,0)$. Then the sum is estimated from above by $$\int\limits_{1\leq x^2+y^2\leq 2N^2} \frac{\C\,dxdy}{x^2+y^2} +\C<\C\int_{r=1}^{2N}\frac{rdr}{r^2}+\C\leq \C\ln N.$$%\add{don't loose constants?}
\end{proof}

\begin{lemma}
\laber{lemma_poisson}
Let $k,\mu\in \NN$. \m{key lemma that bound implies linear} For all $N>\C(k)\mu$ the following holds. Let $F$ be any non-negative integer-valued function on $A=\big([0,N]\times [0,N]\big)\cap \ZZ^2$ satisfying $$|F(v)|\leq \mu(|v|+1).$$ Let $v_1,v_2,\dots v_N$  be points in $\ZZ^2$ (not necessary distinct) and suppose that $G=F+\sum_{k=1}^NG_{v_k}$ (see Definition~\ref{def_poisson}) is a discrete harmonic function on $A$. Then there exists a square of size $k\times k$ in $A$ such that $F$ is linear on this square. 
\end{lemma}

\begin{proof} $$\text{Applying\ Lemma~\ref{lemma_duffin}\ for\ }v\in A'=\left[\frac N5,\frac{4N}5\right]\times\left[\frac N5,\frac{4N}5\right] \text{\ we\ obtain\ $|\partial_\bullet G|\leq \frac{\mu \C N}{N/5}$.}$$  Proceeding as in Lemma~\ref{lemma_harmonic}, we see that in the square $$A''=\left[\frac {2N}5,\frac{3N}5\right]\times\left[\frac {2N}5,\frac{3N}5\right]$$ the second discrete derivatives $\partial_\bullet\partial_\bullet G$ are at most $$\frac{\C\mu }{N}$$ by the absolute value, which is less than $\frac12$ if $N>\C(k)\mu$.

%\begin{figure}[h]
%\includegraphics[width=8cm]{squares.pdf}
%\end{figure}

Since $\sum_{w\in A}\partial_\bullet\partial_\bullet G_{v_k}(w)$ is at most $\C\ln N$ (Lemma~\ref{lemma_log}), we obtain by the direct calculation that $$\sum_{k=1}^N\sum_{w\in A} \partial_\bullet\partial_\bullet G_{v_k}(w)\leq \C N\ln N.$$ 

We cut $A''$ on $(\frac N{5k})^2$ squares of size $k\times k$. Therefore for $N>\C(k)\mu$ we can find a square $A'''\subset A''$ of size $k\times k$ such that  $$\sum_{k=1}^N|\partial_\bullet\partial_\bullet G_{v_i}(v)|\leq 1/3 \text{\ at\ every\ point\  $v\in A'''$.}$$ The estimates for $|\partial_\bullet\partial_\bullet G|$ and $\sum_{i=1}^N|\partial_\bullet\partial_\bullet G_{v_i}|$  imply that for all second derivatives of $F$ we have $\partial_\bullet\partial_\bullet F(v)=0$ for $v\in A'''$. Thus $F$ is linear on $A'''$. 
\end{proof}

%\documentclass{amsart}
%\usepackage{amsthm,amsmath,graphicx,amssymb}
%\usepackage[usenames]{color} %color
%\usepackage{colortbl}
%\usepackage{caption}
%\usepackage{subcaption}
%\usepackage{hyperref}
 
%\newcommand\RR{\mathbb R}
%\newcommand\QQ{\mathbb Q}
%\newcommand\ZZ{\mathbb Z}
%\newcommand\NN{\mathbb N}
%\newcommand\G{\Gamma}
%\theoremstyle{definition}
% \newtheorem{theorem}{Theorem}
% \newtheorem{proposition}{Proposition}
% \newtheorem{definition}{Definition}
% \newtheorem{remark}{Remark}
% \newtheorem{lemma}{Lemma}
% \newtheorem{corollary}{Corollary}
% \newtheorem{conjecture}{Conjecture}
%
%\newcommand\e{\varepsilon}
%
%\begin{document}

\section{Estimates on a cylinder}
\label{sec_cylinder}

\begin{definition}
\laber{def_strip}
Let $p,q\in\ZZ,q>1$. We consider the equivalence relation $(x,y)\sim (x+q,y-p)$ on $\ZZ^2$, it respects the graph structure on $\ZZ^2$, so we define a new graph \m{definition of the cylinder} $$\Sigma = \ZZ^2/\sim, \text{\ where\ }\sim \text{ is generated by }  (x,y)\sim (x+q,y-p).$$ We identify $\Sigma$ with the strip $[0,q-1]\times \ZZ$ where each vertex is connected with its neighbors and, additionally, $(0,y)$ is connected with $(q-1,y-p)$ for all $i\in \ZZ$. The concept of discrete harmonic function also easily descends to $\Sigma$.
\end{definition}

Let $G$ be an integer valued function on $\ZZ^2$ satisfying  $$G(i,j)=G(x+q,y-p)\text{\ for\ all\ } x,y\in\ZZ \text{ and\  fixed } p,q>0.$$ The function $G$ naturally descends to $\Sigma$. Let $G$ be an integer valued superharmonic function on $\Sigma$. Suppose that that $0\leq G(x,y)\leq \C y$ for all $y>0, x\in[0,q-1]$. Suppose also that the number of points $v$ with $\Delta G(v)<0$ is finite and denote $\mathcal{D}=\sum_{v\in\Sigma}\Delta G(v)<0$. 

\begin{lemma}
\laber{lemma_stripestimate} Let $G$ as above and $k> |p|+|q|$. Then, for some $m\leq C({k,|\mathcal{D}|})$ the function $G$ is linear on $$\Sigma'=[0,q-1]\times [m,m+k]\subset \Sigma.$$\m{key lemma for cylinder}
\end{lemma}
\begin{proof}
Choose big $N$. Dissect $[0,q-1]\times[0,N(|\mathcal{D}|+1)]$ on $|\mathcal{D}|+1$ parts \begin{align*}[0,q-1]\times &[0,N] \\ [0,q-1]\times &[N,2N],\text{etc.}\end{align*}
Then there exists  a part $A$ %=[0,q-1]\times[N l,N (l+1)]$ 
 in this dissection where $G$ is discrete harmonic. Note that $$0\leq G|_A\leq \C\cdot (|\mathcal{D}|+1)N.$$ Let $v$ be the center of $A$. Applying Lemma~\ref{lemma_harmonic} for $v$ and $R=N/6$ we prove that derivatives $\partial_\bullet\partial_\bullet G$ are zeros in $B_{N/6}(v)$ if $N>\C|\mathcal{D}|$ and thus $G$ is linear on $B_{N/6}(v)$. If $N/6>2k$ then we found desired $\Sigma'\subset B_{2k}(v)$. So, it was enough to take $\C({k,|\mathcal{D}|} )= \C\cdot (|\mathcal{D}|^2+k)$.
\end{proof}

%We suppose that  the sequence $\{F_n\}_{n=1}^\infty$ (Definition~\ref{def_thetan}) does not stabilize and we will show that this leads to a contradiction. Below we list some lemmata and definitions, critical in the proof of Theorem~\ref{th_stabilfn}, and then we elaborate the details and give all the proofs of the lemmata.

\begin{lemma}
\laber{lemma_periodicsmoothongs}
Let $F=\Psi_\edge$ (see \eqref{eq_edge}). Then for all $n\in\ZZ_{>0}$ smoothings $S_n(F)$ are periodic in the direction $e=(q_1,-p_1)$, i.e. $S_n(F)(v)=S_n(F)(v+e)$ for all $v\in\ZZ^2$.
\end{lemma}
\begin{proof}
Suppose, to the contrary, that $S_n(F)(v)>S_n(F)(v+e)$ for some $v\in\ZZ^2$. It follows from Lemma~\ref{lem_minharmonic} that $\tilde S_n(F)(w)=\min(S_n(F)(w),S_n(F)(w+e))$ belongs to $\Theta_n(F)$, but 
$\tilde S_n(F)(v)<S_n(F)(v)$ which contradicts to the minimality of $S_n(F)$ in $\Theta_n(F)$.
\end{proof}

\begin{lemma}[cf. Lemma~\ref{lem_alldirection}]
\laber{lemma_edge}
Let $F=\min(px+qy,0)$ and $\Sigma$ as above, note that $F$ descends to $\Sigma$. Let $$A\subset \{F\ne S_n(F)\}, A=[0,q-1]\times [m,m+|p|+|q|].$$ Suppose that $S_n(F)$ restricted to $A$ is linear. Then $\mathrm{gcd}(p,q)>1$.
\end{lemma}
\begin{proof}
Since $S_n(F)$ is periodic in the direction $(q,-p)$, we conclude that $S_n(F)(x,y)|_A=k(px+qy)+k'$ for some $k,k'\in\ZZ$. The property of $(0,1)$-increasing implies that $k\geq 0$. \m{established that it is linear}

Suppose that $k=0, S_n(F)=k'$ on $A$. Then $k'<0$ because $S_n(F)|_A< F|_A$. Let $y_0$ be $\max\{y|S_n(F)(1,y)=k'\}$. Then $S_n(F)$ is not superharmonic at $(1,y)$, which is a contradiction. Therefore $k>0$.

Consider the function $F'(x,y)=F(x,y)-px-qy$. Using Remark~\ref{rem_linear}, we write  $$S_n(F')(x,y)=S_n(F)(x,y)-px-qy$$ and repeat verbatim all the above consideration, which gives $k<1$. \m{symmetrical stuff. Would be nice to extract as a lemma}

Since $k(px+qy)$ has integer values and $0<k<1$ we conclude that $\mathrm{gcd}(p,q)>1$.
\end{proof}

\section{Proof of Theorem~\ref{th_stabilfn} for $\Psi_{\edge}$}
\label{sec_proofedge}

\begin{proof}
For the sake of notation denote $F=\Psi_\edge$ (see \eqref{eq_edge}), $p=p_1,q=q_1$. We will prove  that the sequence $\{S_n(F)\}_{n=1}^\infty$ of $n$-smoothings (Definition~\ref{def_thetan}) of $F$ eventually stabilizes. It is easy to check that in the cases when $(p,q)=(\pm 1,0),(0,\pm 1),(\pm 1,\pm 1)$ we have $S_1(F) = F$ (cf. Remark~\ref{rem_cases}). \m{simple cases} Therefore, we conclude the proof of the theorem in this case by Corollary~\ref{cor_fnmin}. From now on we suppose that $pq\ne 0$, $q>1$ without loss of generality, and that the sequence $\{S_n(F)\}_{n=1}^\infty$ does not stabilize.

By Lemma~\ref{lemma_periodicsmoothongs}, all $S_n(F)$ are periodic in the direction $(q,-p)$. Consider the quotient $\Sigma$ of $\ZZ^2$ by translations by $(q,-p)$ (see Definition~\ref{def_strip}). \m{descdening to cylinder}
Abusing notations, we think of $F,S_1(F),S_2(F),\dots $ as functions on $\Sigma$. Note that $\mathcal{D}=\sum_{v\in\Sigma}\Delta F(v)$ is finite. Indeed, $\min(0,px+qy)$ has only finite number of points in $\Sigma$ where the Laplacian is not zero.

Applying Lemma~\ref{lemma_nabla} for a big enough neighborhood of $D(F)$ we observe that $\sum_{v\in\Sigma}\Delta S_1(F)(v)=\mathcal{D}$.  Similarly, we obtain $\sum_{v\in\Sigma}\Delta S_n(F)(v)=\mathcal{D}$ for all $n\in\ZZ_{>0}$ and because of superharmonicity of $S_n(F)$ we see that 
\begin{equation}
|D(S_n(F))|=|\{v\in\Sigma | \Delta S_n(F)(v)\ne 0\}|\leq \mathcal{D}.
\end{equation}

Since the sequence $\{S_n(F)\}_{n=1}^\infty$ does not stabilize, by Corollary~\ref{cor_1smoothing} for each $n\in\ZZ_{>0}$ the set $$A_n=\{v\in\ZZ^2|S_n(F)(v)=F(v)-n\}$$ is not empty. Hence $A_1\supset A_2\supset A_3\dots$, and $A_1$ is finite because $A_1\subset D(F)$ by Corollary~\ref{cor_grows}. \m{taking point in depth} Thus we can take $v_0\in \bigcap\limits_{n\geq 1}A_n$.

Note that $F$ is $(0,1)$-increasing and by Corollary~\ref{cor_monotf} so do all $S_n(F)$. \m{reduction to pq, monotonicity} Also if  $m,k\in\ZZ$ are such that $0\leq m\leq q-1,k>|p|$ then $pm+qk>0$ and consequently all $S_n(F)$ are $(m,k)$-increasing. The property of $(m,k)$-increasing gives that $$F(v_0)-n=S_n(F)(v_0)\geq S_n(F)(v_0-(m,k))$$ which is less than $F(v_0-(m,k))$ for fixed $(m,k)$ and $n>\C k$. Therefore $\supp (F-S_n(F))$ %contains $[0,q-1]\times [-\C n,0]$ for all $n$ big enough, i.e. $\supp (F-S_n(F))$ 
grows at least linearly in $n$. 

For big $n$ let \add{write explicitely} $$c=\min \{F(x,y)|(x,y)\in \supp (F-S_n(F))\}.$$ Applying Lemma~\ref{lemma_stripestimate} to the function $S_n(F)-c$ we note that $S_n(F)$ is linear on $A\subset\supp (F-S_n(F)), A=B_{q+|p|}(v_0)$. We conclude the proof because Lemma~\ref{lemma_edge} implies that $\mathrm{gcd}(p,q)>1$ which contradicts the definition of $\Psi_\edge$.
\end{proof}

\begin{remark}
\label{rem_sumdelta} The following equality holds: %\add{put just after definition of D and that it does not change over time}
 $|\mathcal{D}|=p^2+q^2.$ 
\end{remark}

\begin{proof}
For convenience, consider a function $G(x,y)=\min(0,px-qy)$ and the lattice rectangle $R=[0,q]\times [0,p]\cap\ZZ^2.$ Then 
$$\mathcal{D}=\sum_{R\backslash (q,p)}\Delta G.$$  On the other hand, the sum of Laplacians over the rectangle $R$ is reduced to the sum along its boundary (Lemma~\ref{lemma_nabla}), i.e. 
$$\sum_R \Delta G=\sum_{k=0}^p(G(0,k)-G(-1,k))+\sum_{k=0}^p(G(q,k)-G(q+1,k))+$$
$$+\sum_{k=0}^q(G(k,0)-G(k,-1))+\sum_{t=0}^q(G(k,p)-G(k,p+1)).$$ 
Since $\Delta G(q,p)=-p-q$ we have $$-\mathcal{D}=-p-q-\sum_{R}\Delta G=-p-q+(p+1)p+(q+1)q.$$
\end{proof}

This equality was observed earlier in~\cite{firstsand}. Note also that $p^2+q^2$ is the symplectic area of an edge $(p,q)$ in a tropical curve (see ~\cite{us_series} for details).

%\begin{lemma}
%Inside a pattern there is no infinite territory. Probably, needed when we drop a grain inside an edge.
%\end{lemma}
%\todo{proof}

\begin{corollary}
\laber{cor_smoothingedge}
Let $p,p',q,q',a,a'\in\ZZ$. Suppose that $\mathrm{gcd}(p-p',q-q')=1$.
Then there exists the canonical smoothing $\theta_{p,q,a,p',q',a'}(x,y)$ of $F(x,y)=\min (px+qy+a,p'x+q'y+a')$. Furthermore,$$\theta_{p,q,a,p',q',a'}(x,y)=\theta_{p-p',q-q',a-a',0,0,0}(x,y)=\theta_{p-p',q-q',0,0,0,0}(x+(a-a')p'',y+(a-a')q'')$$ where $(p'',q'')\in\ZZ^2$ satisfies $(p-p')q''+(q-q')p''=1$. %\add{reformulate this statement in human language} 
\end{corollary}

\begin{proof}
The operation $f(x,y)\to f(x,y)+p'x+q'y+a'$ of adding a linear function  commutes with $n$-smoothings and $$\min\big((p-p')x+(q-q')y+(a-a'),0\big) = \min \big((p-p')(x+(a-a')p'')+(q-q')(y+(a-a')q''),0\big).$$
\end{proof}

\subsection{Classification of solitons, proof of Theorem~\ref{th_main}}
\label{proof_thmain}

Consider a movable line-shaped $(p,q)$-periodic state $\phi$ with $q>0$. As in Section~\ref{sec_cylinder} we pass to the cylinder $\Sigma=\ZZ^2/\{(x,y)\sim (x+p,y+q)\}$. Then $\{\phi\ne 3\}$ is contained in $[0,q-1]\times[-k,k]$ for some $k\in \ZZ$, because of line-shapedness. 

\begin{lemma}
In the above setting, if we send a wave from a point $(x,y)\in\Sigma, x\in[0,q-1],y>>0$ then in the set $\{y<<0\}$ there will be no topplings.
\end{lemma}
\begin{proof}
Suppose that there is a toppling in $\{y<<0\}$. Send $n$ such waves where $n$ is big enough. Then the toppling function $F$ will be equal to $n$ in $\{y>>0\}$, $F=n$ in $\{y<<0\}$, $\Delta F \geq 0$ in $[0,q-1]\times[-k,k]$ and $\Delta F\leq 0$ in the position of the soliton after $n$ waves, it is the set $[0,q-1]\times [\C n-k,\C n+k]$. Note that for each $v\in \Sigma$ we have $F(v)\geq n- \C$ because during a wave a point does not topple only if it belongs to the soliton, and the latter moves with some constant speed. Take a point $v_1\in \Sigma$ with $\Delta F(v)<0$. It has a neighbor $v_2$  with $F(v_2)<F(v_1)$ and then $v_2$ has a neighbor $v_3$ with  $F(v_3)< F(v_2)$, etc. Since the distance between  $\{\Delta F<0\}$ and $\Delta F>0$ is al least $\C n$ we will find a point $v$ with $F(v)<n-\C-1$ which is a contradiction.
\end{proof}

Then the toppling function for a soliton is bounded from above by $\min (n, p'x+q'y+r)$ such that $pp'+qq'=\mathrm{gcd}(p,q)$ and $p'x+q'y+r>0$ in $[0,q-1]\times [-k,+\infty]$. Therefore, as in Section~\ref{sec_proofedge} and Corollary~\ref{cor_wavegp} we obtain that $F$ is linear in $[0,q-1]\times [k,k+\C]$ and therefore the soliton is $3+\Delta \theta_{\min (p'x+q'y,0)}$.

\section{Reduction to a smaller state}
We use the notation of Theorem~\ref{th_stabilfn}. Consider $\Psi_\vertex$, \eqref{eq_tripode}. We denote by $\Psi_\vertex'$ the function \m{choose new functions which are better.} %\add{Can we explain it?}
\begin{equation}
\Psi_\vertex'(x,y) = \min \big(\theta_{\min(0,p_1x+q_1y)}(x,y), \theta_{\min(0,p_2x+q_2y+c_1)}(x,y), \theta_{\min(p_1x+q_1y,p_2x+q_2y+c_1)}(x,y)\big).
\end{equation}

Consider $\Psi_\node$, \eqref{eq_node}. We denote by $\Psi_\node'$ the function 
\begin{align}
\Psi_\node'(x,y) = & \min \big(\theta_{\min(0,p_1x+q_1y)}(x,y), \theta_{\min(0,p_2x+q_2y+c_1)}(x,y), \\
 &\theta_{\min(p_1x+q_1y,(p_1+p_2)x+(q_1+q_2)y+c_2)}(x,y), \theta_{\min(p_2x+q_2y+c_1,(p_1+p_2)x+(q_1+q_2)y+c_2)}(x,y) \big).
\end{align}

Note that each of the functions $F=\Psi_\vertex', \Psi_\node'$ is $\C$-holeless for some $\C$, because $D(\Psi_\vertex'), D(\Psi_\node')$ are periodic. Therefore by applying Corollary~\ref{cor_finiteness} we obtain the following remark. 
\begin{remark}
\laber{rem_grows}
Corollary~\ref{cor_grows} holds for $F=\Psi_\vertex', \Psi_\node'$, if $n$ is big enough.
\end{remark}

\begin{lemma}
\label{lemma_changesmoothing}
Let $F$ be $\Psi_\vertex$ (resp. $\Psi_\node$) and $F'$ be $\Psi_\vertex'$ (resp. $\Psi_\node'$). The following conditions are equivalent:
\begin{itemize}
\item The sequence of $n$-smoothings $S_n(F)$ of $F$ stabilizes.
\item The sequence of $n$-smoothings $S_n(F')$ of $F'$ stabilizes.
\end{itemize}
\end{lemma}
\begin{proof}
It is enough to note that $F'$ coincides with $F$ outside of a finite neighborhood of $D(F)$ because we have already proven Theorem~\ref{th_stabilfn} for the case of $\Psi_\edge$. Hence there exists $n$ such that $|F-F'|<n$. Therefore $S_n(F)\leq F'\leq F$ and smoothings of $F'$ can be estimated by smoothings of $F$ and vice versa. 
\end{proof}

We want to consider $F'$ instead of $F$ because of the following lemma.
\begin{figure}[h]
\begin{minipage}{.4\textwidth}
\begin{tikzpicture}[scale = 1.2]
\draw[thick] (-1,-2)--(1,2);
\draw[dashed] (-1.5,-2)--(0.5,2); 
\draw[dashed] (-2,-2)--(0,2); 
\draw[dashed] (-2.5,-2)--(-0.5,2); 
\draw[dashed] (-3,-2)--(-1,2); 
\draw[double] (0,0)--(-3,0);
\draw[double] (0,-0.1)--(-3,-0.1);
\draw (0.2,0.4)--(-0.3,0.5)--(-0.5,0.6)--(-0.8,0.4)--(-1.3,0.4)--(-1.5,0.6)--(-1.8,0.4);
\draw (-0.2,-0.4)--(-0.7,-0.3)--(-0.8,-0.3)--(-1.2,-0.4)--(-1.3,-0.3)--(-1.7,-0.3)--(-1.8,-0.3)--(-2.2,-0.4);

\draw (-1.2,-0.8)--(-1.2,-0.9)--(-2.2,-0.9)--(-2.2,-0.8);
%\draw (1,2)--(2,1)--(2,-0.7)--(-1,-2);
\draw (-0.2,-0.5) node[right]{$p'x+q'y<0$};
\draw (-1.1,-2) node[above]{$Q_1$};
\draw (-1.6,-2) node[above]{$Q_2$};
\draw (-2.1,-2) node[above]{$Q_3$};
\draw (-2.6,-2) node[above]{$Q_4$};
\end{tikzpicture}
Illustration for Lemma~\ref{lem_disruption}. Horizontal   line represents $\{\Delta \phi \ne 0\}$, broken lines along it represent the boundary of $\{G \ne S_1(G)\}$. Slices $Q_2,Q_4$ are identical.
\caption{}
\label{fig_strike}
\end{minipage}\hfill
\begin{minipage}{.1\textwidth}
\end{minipage}
\begin{minipage}{.4\textwidth}
\begin{tikzpicture}[scale=1.2]
\draw[double] (3,0)--(-3,0);
\draw[double] (3,-0.1)--(-3,-0.1);

\begin{scope}
\draw[dashed, thick] (-2,-2)--(0,2); 
\draw[dashed] (-2.5,-2)--(-0.5,2); 
\draw[dashed] (-3,-2)--(-1,2); 
\draw (-0.8,0.4)--(-1.3,0.4)--(-1.5,0.6)--(-1.8,0.4);
\draw (-1.2,-0.4)--(-1.3,-0.3)--(-1.7,-0.3)--(-1.8,-0.3)--(-2.2,-0.4);
\draw (-1.6,-2) node[above]{$Q_2$};
\draw (-2.1,-2) node[above]{$Q_3$};

\end{scope}
\begin{scope}[xshift=1cm]
\draw[dashed, thick] (-2,-2)--(0,2); 
\draw[dashed] (-2.5,-2)--(-0.5,2); 
\draw[dashed] (-3,-2)--(-1,2); 
\draw (-0.8,0.4)--(-1.3,0.4)--(-1.5,0.6)--(-1.8,0.4);
\draw (-1.2,-0.4)--(-1.3,-0.3)--(-1.7,-0.3)--(-1.8,-0.3)--(-2.2,-0.4);
\draw (-1.6,-2) node[above]{$Q_2$};
\draw (-2.1,-2) node[above]{$Q_3$};

\end{scope}
\begin{scope}[xshift=2cm]
\draw[dashed, thick] (-2,-2)--(0,2); 
\draw[dashed] (-2.5,-2)--(-0.5,2); 
\draw[dashed] (-3,-2)--(-1,2); 
\draw (-0.8,0.4)--(-1.3,0.4)--(-1.5,0.6)--(-1.8,0.4);
\draw (-1.2,-0.4)--(-1.3,-0.3)--(-1.7,-0.3)--(-1.8,-0.3)--(-2.2,-0.4);
\draw (-1.6,-2) node[above]{$Q_2$};
\draw (-2.1,-2) node[above]{$Q_3$};

\end{scope}
\begin{scope}[xshift=3cm]
\draw[dashed, thick] (-2,-2)--(0,2); 
\draw[dashed] (-2.5,-2)--(-0.5,2); 
\draw[dashed] (-3,-2)--(-1,2); 
\draw (-0.8,0.4)--(-1.3,0.4)--(-1.5,0.6)--(-1.8,0.4);
\draw (-1.2,-0.4)--(-1.3,-0.3)--(-1.7,-0.3)--(-1.8,-0.3)--(-2.2,-0.4);
\draw (-2.1,-2) node[above]{$Q_3$};
\end{scope}
\end{tikzpicture}
Taking the region in between of $Q_2,Q_4$ we repeat it, thus obtaining a smoothing of $\theta_{px+qy,0}$ which is a contradiction.
\caption{}
\label{fig_smooth}
\end{minipage}

%\caption{sdvsdv}
\end{figure}

\begin{lemma}
\laber{lem_disruption}
Let $(p,q),(p',q')$ be primitive vectors such that $pq'-p'q=1$. Denote $A= \{(x,y)|p'x+q'y\leq 0\}$. Let $G:\ZZ^2\to\ZZ$ be equal to $\theta_{\min(0,px+qy)}$ in the region $p'x+q'y\geq 0$. Then there exists a constant $\C$ such that $$\{G\ne S_1(G)\}\setminus B_1(A) \text{ is contained in } B_\C(A).$$ 
\end{lemma}

\begin{proof}
We know that $\{G\ne S_1(G)\}$ is contained in the union of $B_1(A)$ and $B_1(\Delta G\ne 0)$. Therefore we need to prove that $\{G\ne S_1(G)\}\setminus B_1(A)$ (which is in $1$-neighborhood of $\{\Delta G \ne 0\}$) can not be far from $A$. Suppose the contrary.
\m{no rupture}
The function $G\_{D(G)}$ is periodic, so we can cut $D(G)$ into periodic pieces, Figure~\ref{fig_strike}. We look at $G-S_{1}(G)$ on the periodic pieces of and find two of them with the the same restriction of $G-S_{1}(G)$. Then it means that we could smooth more the initial function $\theta_{\min(0,px+qy)}$: indeed, take all the pieces in between of these two, repeat that all along as in Figure~\ref{fig_smooth}, and decrease $\theta_{\min(0,px+qy)}$ according to $G-S_{1}(G)$ periodically.
\end{proof}

\begin{lemma}\laber{lem_finsupfn}
For all $k\in\ZZ_{\geq 0}$ the cardinality of the set $\{F'\ne S_k(F')\}$ is finite.
\end{lemma}
\begin{proof}%[Proof of Lemma \ref{lem_finsupfn}]
%\label{proof_lemfinsupfn}
Note that $D(F)$ is made of solitons. Each time we apply $1$-smoothing, the set $\{F\ne S_n(F)\}$ belongs to a finite neighborhood of $D(F)$. Therefore we only need to prove  that $\{S_n(F)\ne S_{n+1}(F)\}$ can not propagate far {\bf along} a soliton, which is exactly the assertion in Lemma~\ref{lem_disruption}. 
\end{proof}

\section{Growth of an internal harmonic region}
%\add{why on picture something polygon?}
\begin{definition}
For a subset $A\subset \mathbb Z^2$ we define $r(A)$ as $\max_{(x,y)\in A}(\sqrt{x^2+y^2})$, i.e. the maximal distance between $A$ and $(0,0)$.
\end{definition}

\begin{lemma}\laber{lem_rnupbound}
The sequence $R_n=r\Big(\{F'\ne S_n(F')\}\Big)$ grows at most linearly in $n$, i.e. $R_n\leq \C n$ for all $n\in\ZZ_{>0}$.

\end{lemma}

\begin{proof}
Let $c_n$ be the minimal number such that the support of $F'-S_n(F')$ is contained in $B_{c_n}(O)$.
It is enough to prove that $c_{n+1}\leq c_n+\C$ for all $n$. Now, look at how the support of $F'-S_{n+1}(F')$ differs from the support of $F'-S_n(F')$ outside of $B_{c_n}(O)$. \add{and C depends on what?} It follows from Remark~\ref{rem_grows} that $\supp(F'-S_n(F'))$ belongs to the $n+1$-neighborhood of $\supp(\Delta F')$ for $n$ big enough. \add{what that?}Then we use Lemma~\ref{lem_disruption} along each soliton.
\end{proof}

\begin{lemma}\laber{lem_rnlowbound}
The sequence 
\begin{equation}
\label{eq_rn}
r_n=\max\Big\{r| B_r(O)\subset \{S_n(F')\ne S_{n+1}(F')\}\Big\}
\end{equation}
 grows at least  \add{changed review proof} linearly in $n$, $r_n\geq \C n$ as long as $n$ is big enough.
\end{lemma}
\begin{proof}
We suppose that the sequence $\{S_n(F)\}_{n=1}^\infty$ does not stabilize. Therefore, by Corollary~\ref{cor_1smoothing} for each $n\in\ZZ_{>0}$ the set $A_n=\{v\in\ZZ^2|S_n(F)(v)=F(v)-n\}$ is not empty. Hence $A_1\supset A_2\supset A_3\dots$, and $A_1$ is finite by Lemma~\ref{lem_finsupfn}. Thus we can take $v_0\in \bigcap\limits_{n\geq 1}A_n$. 
  Take any point $v\in \ZZ^2$. Consider the vector $\tilde v=v_0-v$. By adding a suitable linear function to $F$  we may suppose  that $F$ is $\tilde v$-increasing. Hence we may suppose that $F'$ is $\tilde v$-increasing. Then $S_n(F')(v)\leq S_n(F')(v_0)\leq F'(v_0)-n$. Therefore there exists a constant $\C$ (depending on the slopes of the linear parts of $F$) such that if $|v-v_0|< \C n$ then $F'(v)\geq F'(v_0)-n$. For such $v$, clearly, $S_n(F')(v)<F'(v)$ which, with the fact that $|v_0|$ is a fixed finite number, concludes the lemma. 
\end{proof}

\begin{lemma}\laber{lem_linnumberofcolored} There exists a constant $\rho$ such that the number of points $v$ in $B_{n\C}(0,0)$ with $\Delta S_n(F')(v) <0$ is at most $\rho n$ for $n$ big enough.
\end{lemma}

\begin{proof}
The functions $F', S_n(F')$ coincide outside $B_{n\C}(O)$ and superharmonic, therefore it follows from Lemma~\ref{lemma_nabla} that $$\sum_{v\in B_{n\C}(0,0)}|\Delta S_n(F')(v)|=\sum_{v\in B_{n\C}(0,0)}\Delta S_n(F')(v)=\sum_{v\in B_{n\C}(0,0)}\Delta F'(v).$$ Then, outside of a finite neighborhood of $(0,0)$ the function $\Delta F'(v)$ coincide locally with $\Delta \theta_{\min(p'x+q'y,c)}$ in each direction, and $\sum_{v\in B_{n\C}(0,0)}\Delta \theta_{\min(p'x+q'y,c)}$ is linear in $n$ for any coprime $p',q'\in\ZZ^2$. This works both for $\Psi_\vertex$ and $\Psi_\node$.
\end{proof}

\section{Proof of Theorem~\ref{th_stabilfn} for  $F=\Psi_\vertex$ and  $F=\Psi_\node$.}
\label{proof_vertices}
A     geometric explanation of the proof is as follows. Since $r_n$ (see \eqref{eq_rn}) grows linearly, $\{S_n(F')\ne F'\}$ encircles a figure with the area of order $n^2$. In this figure, $\{\Delta S_n(F')\ne 0\}$ is of linear size, hence we can find a big part where $S_n(F')$ is harmonic and with at most linear growth, thus, it is linear, which will contradict to Lemma~\ref{lem_alldirection}.

Now we supply all the details. Suppose that the sequence $\{F_n\}$ of $n$-smoothings of $F$ does not stabilize as $n\to\infty$. Therefore, by Lemma~\ref{lemma_changesmoothing} the sequence  of  $\{S_n(F')\}$ of $n$-smoothings of $F'$ does not stabilize. Lemma~\ref{lem_finsupfn} asserts that the support of $F'-S_n(F')$ is finite, and Lemmata~\ref{lem_rnupbound},\ref{lem_rnlowbound} tell us that the set $\{F'\ne S_n(F')\}$ grows at most and at least linearly in $n$. Refer to Figure~\ref{fig_smoothing}: the grey region is $\{F'\ne S_n(F')\}$, internal (resp. external) circle has radius $r_n=\C n$, (resp. $R_n=\C n$ with other $\C$) and represents a subset (resp. superset) of $\{F'\ne S_n(F')\}$. 

Remark~\ref{rem_cases} eliminated several simple cases. \add{non-esthetic change of coordinates}So, after change of coordinates $x\to y, y\to x$, if necessary,  we may assume that the tropical curve defined by $F$ does not contain the vertical ray (the dashed line on Figure~\ref{fig_smoothing}) and  the top part of the dashed line belongs to the region where $F\equiv 0$. \add{it is not clear} Note that $F$ is $(0,1)$-increasing (Definition~\ref{def_monotone}), and, thus, $F'$ and all $S_n(F')$ are $(0,1)$-increasing by Corollary~\ref{cor_monotf}.

By Lemma~\ref{lem_linnumberofcolored} the number of points $v$ in $B_{n\C_6}(O)$ with $\Delta S_n(F')(v)<0$ is bounded from above by $\rho n$ for some fixed $\rho$. We draw a rectangular $R$ like\add{delete like} in Figure~\ref{fig_smoothing}. Namely, the vertical sides of $R$ lie on different sides of the dashed line $x=0$, the horizontal sides has length $cN$ with some fixed $c$ \add{who is c?}, $R$ does not intersect the set $\Delta F'<0$ outside of $B_{\C_6n}(O)$.

Take $N$ big enough and consider $S_N(F')$. Then we choose $l$ big enough and cut $R$ into squares of size $\frac{cN}{l}$ (later we refer to them as small squares). We consider the intersection $R'=R\cap\{F'\ne S_N(F')\}$ and pick all the small square which belongs to this intersection.  

Comparing the area of $R'\sim N^2$ with $\rho N$ we see that there exists a small square $S$ which contains at most $\frac{cN}{k}$ points $v$ with $\Delta S_N(F')(v)<0$. Let $M$ be the minimum of $S_N(F')$ on $R'$. 
Then Lemma \ref{lemma_poisson} implies that $S_n(F')-M$ should be linear on this small square $k\times k$ which is a subset of $S$. This, by Lemma~\ref{lem_alldirection}, is a contradiction.

\begin{figure}[h]
\begin{center}
\begin{tikzpicture}
\begin{scope}[scale=1.2]
%\draw[fill=black!80, circular drop shadow] (0,0) circle (1.5cm);
\pgfmathsetseed{2}
\begin{axis}[xmin=-6, xmax=6, ymin=-6, ymax=6, hide axis]
%\draw[step=1, black, very thick ] (0,0) grid (10,5);
\addplot[line width=1pt, color=black, name path = f1,domain=-2.5:2.5, samples=64, smooth] (\x,{sqrt(6.25-\x*\x)+0.1*rand*(6.25-\x*\x)});
\addplot[line width=1pt, color=black, name path = f2,domain=-2.5:2.5, samples=64, smooth] (\x,{-sqrt(6.25-\x*\x)-0.1*rand*(6.25-\x*\x)});
%\addplot [
%        thick,
%        color=gray,
%        fill=gray!30, 
%        %fill opacity=0.55
%    ]
%fill between[
%        of=f1 and f2,
%        %soft clip={domain=-1:0},
%    ];
\end{axis}
    
\end{scope}
\begin{scope}[scale=0.6, xshift=6.7cm, yshift=5.5cm]
\draw (0,0) circle (1.9);

\draw (0,0) circle (3.2);

\draw[very thick] (0,0) -- (-2.5,-7);
\draw[very thick] (0,0) -- (3,-7.5);
\draw[very thick] (0,0) -- (1,8);
\draw[dashed] (0,5)--(0,-5);

\draw (-0.5,-4)--(-0.5,1);
\draw (-0.5+0.2,-4)--(-0.5+0.2,1);
\draw (-0.5+0.4,-4)--(-0.5+0.4,1);
\draw (-0.5+0.6,-4)--(-0.5+0.6,1);
\draw (-0.5+0.8,-4)--(-0.5+0.8,1);
\draw (-0.5+1,-4)--(-0.5+1,1);
\draw (-0.5+1.2,-4)--(-0.5+1.2,1);

\draw (-0.5,1)--(0.7,1);
\draw (-0.5,0.8)--(0.7,0.8);
\draw (-0.5,0.6)--(0.7,0.6);
\draw (-0.5,0.4)--(0.7,0.4);
\draw (-0.5,0.2)--(0.7,0.2);
\draw (-0.5,0)--(0.7,0);
\draw (-0.5,-0.2)--(0.7,-0.2);
\draw (-0.5,-0.4)--(0.7,-0.4);
\draw (-0.5,-0.6)--(0.7,-0.6);
\draw (-0.5,-0.8)--(0.7,-0.8);
\draw (-0.5,-1)--(0.7,-1);
\draw (-0.5,-1.2)--(0.7,-1.2);
\draw (-0.5,-1.4)--(0.7,-1.4);
\draw (-0.5,-1.6)--(0.7,-1.6);
\draw (-0.5,-1.8)--(0.7,-1.8);
\draw (-0.5,-2.0)--(0.7,-2.0);
\draw (-0.5,-2.2)--(0.7,-2.2);
\draw (-0.5,-2.4)--(0.7,-2.4);
\draw (-0.5,-2.6)--(0.7,-2.6);
\draw (-0.5,-2.8)--(0.7,-2.8);
\draw (-0.5,-3.0)--(0.7,-3.0);
\draw (-0.5,-3.2)--(0.7,-3.2);
\draw (-0.5,-3.4)--(0.7,-3.4);
\draw [<->] (-0.5,-5)--(0.7,-5);
\draw (0,-5.5) node {$\sim N$};
\end{scope}
\end{tikzpicture}
\caption{An illustration for the proof of Theorem~\ref{th_stabilfn}.}
\label{fig_smoothing}
\end{center}
\end{figure}
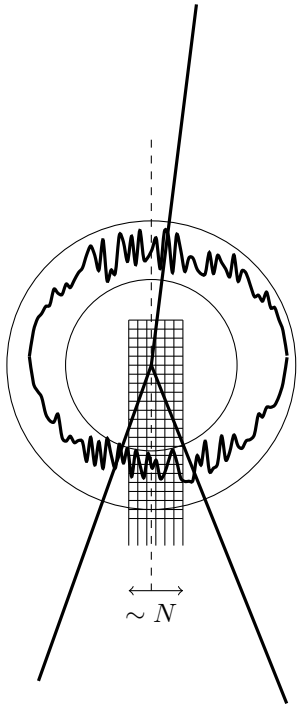

\section{Discussion}
%\subsection{Relation to the other convergences in sandpiles} We defined the procedure of smoothing (Definition~\ref{def_thetan}), and found several properties of smoothing of piece-wise linear functions.  It seems that the pictures in~\cite{levine2013apollonian} are the smoothings of piece-wise quadratic functions, mixed with linear functions. In order to prove that one should generalize the definition and properties of $\ZZ$-increasing (which means more or less monotonicity) (Definition~\ref{def_monotone}, Lemma~\ref{lem_monot1}) for higher discrete derivatives.\add{what is the point? smoothing of non-linear functions}

\subsection{Sand dynamic on tropical varieties, divisors.}

%\add{here the idea is that zero-dimensional solitons - points - behave under the wave action very simple, but ew have more.}
Let $G$ be a graph and $V=\{v_1,v_2,\dots, v_k\}$ be a collection of some of its vertices. Consider the following state $\phi_V = \sum_{v\in G} v\cdot (\mathrm{deg}(v)-1) - \sum_{v\in V}\delta_v$. It corresponds to the divisor $V=\sum_{v\in V}\delta_v$. Let $P=\{p_1,p_2,\dots, p_n\}$ be another collection of vertices of $G$. There there exists a divisor linearly equivalent to $V$ and containing $P$ if and only if the relaxation of $\phi_V+\sum_{p\in P}\delta_p$ terminates.\add{is it understanadble at all? give links to similar one-dim discussion, for example, book of perkinson.}

In this article we studied sandpile on $\h\ZZ^2\cap\Omega$. We can produce the same type of problems for
a tropical variety, if we have a sort of grid on it. The convergence results are expected to be formulated and proven in the same
way.

What is an interesting aspect of the possible applications is the tropical divisors. Indeed, using relaxation in
sand dynamic we can understand if there exist a divisor linearly equivalent to a given tropical divisor $L$, passing through prescribed set of points. For that we represent $L$ as a collection of sand-solitons glued with help of sand triads (because locally $L$ looks like a tropical plane curve and we know which sand-solitons should we take), then we add sand to the points $p_1,p_2,\dots,p_n$ and
relax the obtained state. If the relaxation terminates, it produces
the divisor which is linearly equivalent to $L$. If not, that means that such
a divisor does not exist.

%\add{to the introduction}

\subsection{Application of our methods: fractals and patterns in sandpile}
%\add{when theorems verified}
The sandpile on $\ZZ^2$ exhibits a fractal structure; see, for example, the pictures of the identity element in the critical group~\cite{LP}. As far as we know, only a few cases have a rigorous explanation.\m{fractals} It was first observed in~\cite{ostojic2003patterns} that if we rescale by $\sqrt n$ the result of the relaxation of the state with $n$ grains at $(0,0)$ and zero grains elsewhere, it weakly converges as $n\to\infty$. Than this was studied in~\cite{levine2009strong} and was finally proven in~\cite{PS}. However the fractal-like pieces of the limit found their explanation later, in~\cite{levine2013apollonian, LPS} and happen to be curiously related to Apollonian circle packing. Recently the stability of patterns was proven in \cite{pegden2017stability}.

Our research approaches these questions from the combinatorial and polyhedral points of view and rigorously explains appearance of thin balanced graphs (which are called defects in \cite{pegden2017stability}): in most of these fractal pictures one can find solitons which propagate along such graphs which were observed and studied from the point of view of experimental physics in~\cite{CPS}, Fig.~3,~\cite{firstsand}, see also a thesis~\cite{book} for a one nicely written consistent piece of information. 

In particular, we prove well known (experimentally) fact that solitons move changeless under the action of waves, this is why we call them {\it solitons}. We also study local interactions of solitons: {\it triads} and what is know in tropical geometry as {\it nodes}. A lot of work is to be done in future. Another thesis,~\cite{sadhu2017emergence}(see also~\cite{sadhu2012pattern}), contains a lot of pictures and examples with apparent piece-wise linear behavior. \m{a lot of piecewise linear behavior} We expect that the methods of this article will be used to study the fractal structure in those cases.

\appendix
\section{Locally finite relaxations and waves}
\label{sec_locallyrelaxations}
In this section we study the relaxations and stabilizability issues. The main goal here is to establish The Least Action Principle (Proposition~\ref{prop_leastaction}, cf.~\cite{FLP}) and wave decomposition \m{here we discuss stiabilizability issues}(Proposition~\ref{prop_decompositioninwaves} and Corollary~\ref{cor_leastforwaves}) for {\bf locally-finite} relaxations (Definition~\ref{def_locfin}) on infinite graphs. We also prove that given a finite upper bound on a toppling function of a state, there exists a relaxation sequence of this state which converges pointwise to a stable state (Lemma~\ref{lem_relcritlf}).  

The proofs are the same as in the finite case, but in the absence of references we give all the details here. Sandpiles on infinite graphs were previously considered, for example, in~\cite{MR2150146, fey2009stabilizability, bhupatiraju2016inequalities}, but only from the distribution point of view: in their approach the relaxation (after adding a grain to a random configuration in a certain class) is locally finite almost sure with respect to a certain distribution. However, there are a lot of similarities between this section and~\cite{MR3350375}. \m{add that nothing new, and proof of above estimate }

\subsection{The Least Action Principle for locally finite relaxations, relaxability}
Let $\Gamma$ be at most countable set, $\tau\colon\Gamma\rightarrow\mathbb{Z}_{> 0}$ be a {\it threshold} function and $\gamma\colon\Gamma\rightarrow 2^\Gamma$ be a set-valued function such that 
\begin{itemize}
\item $v\not\in\gamma(v),$ 
\item if $v\in\gamma(w)$, then $w\in\gamma(v)$,
\item $|\gamma(v)|\leq\tau(v) \text{ for all } v\in\Gamma,$ where $|\gamma(v)|$ denotes the number of elements in the set $\gamma(v).$ \add{but all our results are only for planar geaphs, non?}
\end{itemize} 

We interpret $\gamma(v)$ as the set of neighbors of a point $v\in\Gamma.$ We write $u\sim v$ instead of $u\in \gamma(v)$, because $\gamma$ indeces a symmetric relation. The Laplaian $\Delta$ is the operator on the space $\mathbb{Z}^\Gamma=\{\phi\colon\Gamma\rightarrow\mathbb{Z}\}$ of {\it states} on $\Gamma$ given by $$\Delta\phi(v)=-\tau(v)\phi(v)+\sum_{u\sim v}\phi(u).$$ 
A function $\phi$ is called {\it superharmonic} if $\Delta\phi\leq 0$ everywhere. 

\begin{remark}
\laber{rem_suph1} 
Note that $|\gamma(v)|\leq\tau(v) \text{ for all } v\in\Gamma$ holds if and only if the function $\phi\equiv 1$ is superharmonic.
\end{remark}

\begin{example}
In our main situation, $\Gamma$ is a subset of $\ZZ^2$ and $|\gamma(v)|=\tau(v)=4$ for all $v\in\Gamma\setminus \partial\G$. In this case we obtain the standard definition of a Laplaian on $\G\setminus\partial\G$:
\begin{equation}
\label{def_laplacian}
\Delta\phi(v)=-4\phi(v)+\sum_{u\sim v}\phi(u).
\end{equation}
\end{example}

\begin{definition}
 For a point $v\in\Gamma$, we denote by $T_v$ the {\it toppling} operator acting on the space of states $\mathbb{Z}^\Gamma$. It is given by $$T_v\phi=\phi+\Delta\delta(v),$$ where $\delta(v)$ is the function on $\Gamma$ taking $1$ at $v$ and vanishing elsewhere.  
\end{definition}

\begin{definition}\laber{def_relaxseq}
A {\it relaxation} $\phi_\bullet$ of a state $\phi\in\mathbb{Z}^\Gamma$ is a sequence of functions $\phi_\bullet=\{\phi_i\}_{i \in I}$, (for $I=\ZZ_{\geq 0}$ or $I=\{0,1,\dots, n\}, n\in\ZZ_{\geq 0}$) such that $\phi_0=\phi$ and for each $k\geq 0$ there exists $v_k\in\Gamma$ such that $\phi_k(v_k)\geq\tau(v_k)$ and $\phi_{k+1}=T_{v_k}\phi_k.$ The {\it toppling function} $H_{\phi_\bullet}\colon\Gamma\rightarrow\mathbb{Z}_{\geq 0}\cup\{\infty\}$ of the relaxation $\phi_\bullet$ is given by $$H_{\phi_\bullet}=\sum\limits_{i\in I}\delta({v_i}),$$ it counts the number of topplings at every point during this relaxation. We refer to $\{v_i\}_{i\in I}$ as a {\it relaxation sequence}.
\end{definition}

\begin{definition}\laber{def_locfin}
A relaxation $\phi_\bullet$ is called {\it locally-finite} if $H_{\phi_\bullet}(v)$ is finite for every $v\in\Gamma.$ The {\it result of a locally-finite relaxation} is the state $\phi'$ given by the point-wise limit $$\phi'=\phi_0+\Delta H_{\phi_\bullet}=\lim_{k\rightarrow\infty}\phi_k.$$
\end{definition}

\begin{lemma}\laber{lem_boundt}
Consider a locally-finite relaxation $\phi_\bullet$ for a state $\phi$ and a function $F\colon\Gamma\rightarrow\mathbb{Z}_{\geq 0}$ such that $\phi+\Delta F<\tau$. Then $H_{\phi_\bullet}(v)\leq F(v)$ for all $v\in\Gamma$. 
\end{lemma}

\begin{proof}
We use the notation from Definition \ref{def_relaxseq}. Consider the relaxation $\phi_\bullet$ and the corresponding sequence of functions $H_n$ for $n=1,\dots$ given by  

\begin{equation}
\label{eq_topplingn}
H_n=\sum_{i=1}^n\delta({v_i}).
\end{equation}
 Let $H_0\equiv 0$. %Note that if the relaxation sequence is finite, then $H_n$ is equal to $H_{\phi_\bullet}$ for some $n$. If $\phi_\bullet$ is not finite, then $H_{\phi_\bullet}$ is the pointwise limit of $H_n$. Therefore, 
It suffices to show that $H_n\leq F$ for every $n$, and $H_0\equiv 0\leq F$. Suppose that $n>0$ and $H_{n-1}\leq F.$ Since $H_n=H_{n-1}+\delta({v_n}),$ it is enough to show that $H_{n-1}(v_n)<F(v_n).$ We know that $\phi_n(v_n)\geq\tau(v_n)$ and $\phi_n(v_n)=\phi_0(v_n)+\Delta H_{n-1}(v_n).$ Therefore, \begin{align*}\tau(v_n)\leq & \phi_0(v_n)-\tau(v_n)H_{n-1}(v_n)+\sum_{u\sim v_n}H_{n-1}(u)\leq\\
\leq &  \phi_0(v_n)-\tau(v_n)H_{n-1}(v_n)+\sum_{u\sim v_n}F(u)=\\
=&\phi_0(v_n)+\Delta F(v_n)+\tau(v_n) \big(F(v_n)-H_{n-1}(v_n)\big). \end{align*}
Since $\phi_0(v_n)+\Delta F(v_n)< \tau(v_n)$ (by the hypothesis of the lemma) and $\tau(v_n)>0,$ we conclude that
$$1\leq F(v_n)-H_{n-1}(v_n).$$\end{proof}

\begin{corollary}\label{cor_alllf}
Consider a state $\phi.$ If there exist a function $F\colon\Gamma\rightarrow\mathbb{Z}_{\geq 0}$ such that $\phi+\Delta F<\tau$, then all relaxation sequences of $\phi$ are locally finite.  %\add{it is a proof of something else}
\end{corollary}

\begin{lemma}\laber{lem_mergerel}
Consider a state $\phi$ and the set $\Psi$ of all its relaxations $\psi_\bullet$. Then there exists a relaxation $\phi_\bullet$ of $\phi$ such that $$H_{\phi_\bullet}(v)=\sup_{\psi_\bullet\in\Psi} H_{\psi_\bullet}(v), \forall v\in\Gamma.$$
\end{lemma}

\begin{proof}
Consider the set $W=\{(v,k)\}\subset \G\times \ZZ_{\geq 0}$ which contains all pairs $(v,k)$ such that there exists a relaxation $\phi_\bullet^{v,k}\in\Psi$ which has $k$ topplings at the vertex $v\in\G$. Clearly, if $(v,k)\in W, k>0$ then $(v,k-1)\in W$. The set $W$ is at most countable, so we order it as $\{(v_n,k_n)\}_{n=1,2,\dots}$ in such a way that $(v,k-1)$ appears earlier than $(v,k)$ for all $(v,k)\in W, k>0$.

Take any relaxation $\phi_\bullet$. We construct relaxations $\phi_\bullet^0,\phi_\bullet^1,\dots$ in such a way that $\phi_\bullet=\phi_\bullet^0$, all $\phi_\bullet^{\geq n}$ coincide at first $n$ topplings, and for each $n\geq 0$ the toppling function of $\phi_\bullet^n(v_n)$ is at least $k_n$.

Let $\phi_\bullet^{n-1}$ be already constructed, $n\geq 1$, we construct $\phi_\bullet^n$ as follows.

If the toppling function of $\phi_\bullet^{n-1}$ at $v_n$ is at least $k_n$, we are done.  If not, take $\phi_\bullet^{v_{n},k_n}$ and consider its toppling functions $H_{\phi_\bullet^{v_{n},k_n}}^i$ as in \eqref{eq_topplingn} except that we put the bottom index to the top. Take the first $i$ such that there exists $w\in\G$ such that $H_{\phi_\bullet^{n-1}}(w)< H_{\phi_\bullet^{v_{n},k_n}}^i(w)$. Since it is the first such moment, for some $j$ we have $$H^j_{\phi_\bullet^{n-1}}(w')\geq H_{\phi_\bullet^{v_{n},k_n}}^i(w')$$ for all $w'\sim w$. So we add to $\phi_\bullet^{n-1}$ the toppling at $w$ somewhere after $j$-th toppling, and denote the obtained relaxation sequence as $\phi_\bullet^{n-1}$ again. Note that by repeating this cycle of arguments a finite number of times, we will have that $\phi_\bullet^n(v_n)\geq k_n$.
\end{proof}

\begin{definition}
\laber{def_relaxable}
A state $\phi$ is called {\it stable} if $\phi<\tau$ everywhere. A state $\phi$ is called {\it relaxable} if there exist a locally-finite relaxation $\phi_\bullet$ of $\phi$ such that $\phi'$(Definition~\ref{def_locfin}) is stable. Such a relaxation $\phi_\bullet$ is called {\it stabilizing}.
\end{definition}

\begin{corollary}\label{cor_untopf}
If $\phi$ is relaxable, then $H_{\phi_\bullet^1}=H_{\phi_\bullet^2}$ for any pair of stabilizing relaxations $\phi_\bullet^1$ and $\phi_\bullet^2$ of $\phi.$ In particular, $(\phi_\bullet^1)^\circ=(\phi_\bullet^2)^\circ.$
\end{corollary}
\begin{proof}
Applying Lemma \ref{lem_boundt} twice, we have $H_{\phi_\bullet^1}\leq H_{\phi_\bullet^2}$ and $H_{\phi_\bullet^1}\geq H_{\phi_\bullet^2}.$
\end{proof}

\begin{lemma}\laber{lem_relcritlf}
If all relaxations of a state $\phi$ are locally-finite, then $\phi$ is relaxable.
\end{lemma}

\begin{proof}
Consider a point $v\in\Gamma.$ We will prove that there exist $N>0$ such that $H_{\phi_\bullet}(v)<N$ for all relaxations $\phi_\bullet$ of $\phi.$ Suppose the contrary. Then there exists a sequence of relaxations $\phi^n_\bullet$ such that $\lim_{n\rightarrow\infty}H_{\phi^n_\bullet}(v)=\infty.$ Applying Lemma \ref{lem_mergerel} to the sequence $\phi^n_\bullet$ we see that there exists a relaxation of $\phi$, that is not locally-finite.

Therefore, for any $v\in\Gamma$ there exist a relaxation $\phi^v_\bullet$ such that $H_{\phi_\bullet}(v)\leq H_{\phi^v_\bullet}(v)$ for all relaxations $\phi_\bullet$ of $\phi.$ Applying  Lemma \ref{lem_mergerel} again to the family of relaxations $\{\phi^v_\bullet\}_{v\in\Gamma}$ we find a relaxation sequence $\tilde\phi_\bullet$ such that $H_{\phi_\bullet}(v)\leq H_{\tilde\phi_\bullet}(v)$ for all relaxations $\phi_\bullet.$

We claim that $\tilde\phi_\bullet$ is a stabilizing relaxation. Suppose that $\phi+H_{\tilde\phi_\bullet}$ is not stable, i.e. there exists $v\in\Gamma$ such that $\phi(v)+H_{\tilde\phi_\bullet}(v)\geq \tau(v).$ Therefore, we can make an additional toppling at $v$ after the moment when all the topplings at $v$ and its neighbors in $\tilde\phi_\bullet$ are already made. This contradicts to the maximality of $\tilde\phi_\bullet$.
\end{proof}

\begin{proposition}\laber{prop_relcrit}
A state $\phi$ is relaxable if and only if there exists a function $F\colon\Gamma\rightarrow\mathbb{Z}_{\geq 0}$ such that $\phi+\Delta F<\tau$.
\end{proposition}

\begin{proof}
If $\phi$ is relaxable then we can take $F$ to be $H_\phi.$ On the other hand, if such $F$ exists, then by Lemma \ref{lem_boundt} all the relaxations of $\phi$ are locally-finite. Therefore,  $\phi$ is relaxable by Lemma \ref{lem_relcritlf}.
\end{proof}

\begin{definition} 
\laber{def_topplingfunc}
Consider a relaxable state $\phi.$ Denote by $H_\phi$ {\it the toppling function of} $\phi$, where $H_\phi$ is a toppling function of some stabilizing relaxation of $\phi.$ Define {\it the relaxation of} $\phi$ to be the state $\phi^\circ=\phi+\Delta H_\phi.$
\end{definition}

\begin{proposition}[The Least Action Principle, \cite{FLP}]\laber{prop_leastaction}
Let $\phi$ be a relaxable state and $F\colon\Gamma\rightarrow\mathbb{Z}_{\geq 0}$ be a function such that $\phi+\Delta F$ is stable. Then $H_\phi\leq F.$ In particular, $H_\phi$ is the {\bf pointwise} minimum of all such functions $F$. 
\end{proposition}

\begin{proof}
Straightforward by Lemma \ref{lem_boundt}.
\end{proof}

\begin{lemma}
\laber{lemma_onewavetoppling}
Consider a stable state $\phi$ and a point $v\in\Gamma.$ Then the state $T_v\phi$ is relaxable. \vv
\end{lemma}

\begin{proof}
Consider a function $F(z)=1-\delta(v)$ for every $z\in \G$. Then $T_v\phi+\Delta F=\phi+\Delta\delta(v)+\Delta(1-\Delta\delta(v))=\phi+\Delta 1.$ Applying Remark~\ref{rem_suph1} we see that  $T_v\phi+\Delta F$ is stable. Thus, $T_v\phi$ is relaxable by Proposition~\ref{prop_relcrit}.
\end{proof}

%\begin{proposition} Let $\Gamma$ be $\ZZ^2\setminus\{(0,0)\}$. Let $P\subset\Gamma$ be a finite collection of vertices of $\Gamma$. Then the state $\langle 3\rangle + \sum_{p\in P}\delta_p$ is relaxable.
%\end{proposition}
%\begin{proof}
%\add{proof}
%\end{proof}

\subsection{Waves, their action}
\label{sec_waves}

Sandpile waves were introduced in~\cite{ivashkevich1994waves}, see also~\cite{ktitarev2000scaling}.

\add{check how wave function is used in the main part}
\begin{definition}\laber{def_wave}
Let $v$ be a point in $\Gamma.$ The {\it wave} operator $W_v$, acting on the space of the stable states on $\Gamma$, is given by $$W_v\phi=(T_v\phi)^\circ.$$ The {\it wave-toppling} function $H^v_\phi$ of $\phi$ at $v$ is given by 
\begin{equation}
\label{eq_topplewave}
H^v_\phi=\delta(v)+H_{T_v\phi}.
\end{equation}

\begin{remark}
Note that if $v$ has $3$ grains and has a neighbor in $\G\setminus\partial\G$ with $3$ grains, then the result $W_v\phi$ is also a stable state. 
\end{remark}
Indeed, $T_v\phi$ has $-1$ grain at $v$, but the neighbor of $v$ has $4$ grains and will topple. So, eventually, we will have non-negative amount of sand at $v$.
\end{definition}
\begin{remark}\label{rem_wavet}
It is clear that $W_v\phi=\phi+\Delta H^v_\phi.$
\end{remark}

\begin{corollary}[\cite{redig}]
\laber{cor_waveval}
For any $u\in\Gamma$ the value $H^v_\phi(u)$ is either $0$ or $1$. Furthermore, $H^v_\phi(v)=1.$
\end{corollary}

\begin{proof}
It follows from the proof of Lemma~\ref{lemma_onewavetoppling} that $H_{T_v\phi}\leq 1-\delta(v).$
\end{proof}

\begin{lemma}\laber{lem_wavetoplb}
Suppose that $\phi$ is a stable state and $v$ a point in $\Gamma$. If $\phi+\delta(v)$ is relaxable and not stable, then the toppling function for the wave from $v$ is less or equal than the toppling function for a relaxation of $\phi+\delta(v)$, i.e. $$H^v_\phi(w) \leq H_{\phi+\delta(v)}(w) \ , \forall w\in\G\setminus\partial\G.$$
\end{lemma}
\begin{proof}
It is clear that $(\phi+\delta(v))(w)=\phi(w)<\tau(w)$ for all $w\neq v$ and $(\phi+\delta(v))(v)=\tau(v).$ Therefore, $T_v$ is the first toppling in any non-trivial relaxation sequence for $\phi+\delta(v)$ and $H_{\phi+\delta(v)}(v)\geq 1.$ In particular, the function $H_{\phi+\delta(v)}-\delta(v)$ is non-negative and $H_{T_v\phi}\leq H_{\phi+\delta(v)}-\delta(v)$  by Lemma \ref{lem_boundt} since $$T_v\phi+\Delta\big(H_{\phi+\delta(v)}-\delta(v)\big)=\phi+\Delta\delta(v)+\Delta\big(H_{\phi+\delta(v)}-\delta(v)\big)=\phi+ \Delta H_{\phi+\delta(v)}   =\big(\phi+\delta(v)\big)^\circ-\delta(v)<\tau.$$
\end{proof}

\begin{definition}
\laber{def_partial}
Let $\phi$ be a relaxable state, $H_\phi$ be its toppling function. Let $0\leq F\leq H_\phi$. The state $\phi+\Delta F$ is called a {\it partial relaxation} of $\phi$.
\end{definition}

\begin{lemma}\laber{lem_subtoppl}
Consider a relaxable state $\phi$ and an integer-valued function $F$ on $\Gamma$ such that  $0 \leq F\leq H_\phi.$ Then the state $\phi+\Delta F$ is relaxable and $$H_{\phi+\Delta F}=H_\phi - F.$$
\end{lemma}

\begin{proof}
By Proposition \ref{prop_relcrit} the state $\phi+\Delta F$ is relaxable because  $$\phi+\Delta F+\Delta(H_\phi-F)=\phi+\Delta H_\phi=\phi^\circ<\tau$$ and $H_\phi-F$ is non-negative. In particular, $H_\phi-F \geq H_{\phi+\Delta F}$ by Lemma \ref{lem_boundt}. On the other hand, since $H_{\phi+\Delta F}+F\geq 0$, we have $$\phi+\Delta(H_{\phi+\Delta F}+F)=\phi+\Delta F+\Delta H_{\phi+\Delta F}=(\phi+\Delta F)^\circ<\tau.$$ Applying again Lemma \ref{lem_boundt}, we have $H_\phi\leq H_{\phi+\Delta F}+F.$
\end{proof}

\begin{proposition}
\laber{prop_decompositioninwaves}
Let $\phi$ be a stable state and $v$ be a point in $\Gamma.$ Suppose that $\phi+\delta(v)$ is relaxable. Then the relaxation of $\phi+\delta(v)$ can be decomposed into sending $n$ waves from $v$, i.e.
$$(\phi+\delta(v))^\circ=\delta(v)+W^n_v\phi,$$ where $n=H_{\phi+\delta(v)}(v)$ and $W^n_v(\phi) = W_v(W_v(\dots (\phi))\dots)$, $n$-th power of $W_v$. On the level of toppling functions, this gives $$H_{\phi+\delta(v)}=\sum_{k=0}^{n-1}H^v_{(W_v^k\phi)}.$$ 

\end{proposition}
Added parenthesis in the subscript are for better readability only.
\begin{proof}
Combining Lemmata \ref{lem_wavetoplb} and \ref{lem_subtoppl} we have $$H_{\phi+\delta(v)}=H^v_\phi+H_{(W_v\phi+\delta(v))}.$$
If the state $W_v\phi+\delta(v)$ is not stable, then we can apply the same lemmata again. We complete the proof by iteration of this procedure and using Corollary \ref{cor_waveval} (each wave has one toppling at $v$, therefore we have $n$ waves).
\end{proof}

%\begin{lemma}
%If $\phi$ is a stable state and $v\in\Gamma$ then $H^v\phi=H^v(\phi-\delta_v).$
%\end{lemma}
%\add{proof}

\begin{lemma}
\laber{lemma_wavesindependent}
If $\phi$ is a stable state and $v_1,\dots,v_m$ are vertices of $\Gamma$ such that $v_i$ is adjacent to $v_{i+1}$ and $\phi(v_i)=\tau(v_i)-1$ for all $i=1,2,\dots,m$, then $H^{v_1}_\phi=H^{v_m}_\phi$. 
\end{lemma}
\begin{proof}
It follows from the simplest case $m=2$, for which it is just a computation.
\end{proof}

\begin{definition}
\laber{def_waveterritory} 
In a given state $\phi$, a {\it territory} is a maximal by inclusion connected component of the vertices $v$ such that $\phi(v)=\tau(v)-1$. Given a territory $\TT$, we denote by $W_{\TT}$ the wave which is sent from a point in $\TT$ (by Lemma~\ref{lemma_wavesindependent} it does not matter from which one).
\end{definition}

Basically, Corollary~\ref{cor_waveval} tells us that a wave from $v$ increases the toppling function exactly by one in the territory to which $v$ belongs to, and by at most one in all other vertices.

\begin{proposition}\laber{prop_waveleast}
Let $\phi$ be a stable stable, $v$ be a point in $\Gamma$, and $F\colon\Gamma\rightarrow\mathbb{Z}_{\geq 0}$ be a function such that $F(v)\geq 1$ and $\phi+\Delta F$ is stable. Then $F\geq H^v_\phi.$ 
\end{proposition}
\begin{proof}
Similar to Lemma~\ref{lem_boundt}.
\end{proof}
%\add{probably we need here some local notation which we do not need in the main text?}
\begin{corollary}[Least Action Principle for waves, cf\cite{FLP}]
\laber{cor_leastforwaves}\add{writi it to the beginning?}
Suppose that a state $\phi$ is stable. We send $n$ waves from a vertex $v$. Let $H = \sum_{k=0}^{n-1}H^v_{(W^k\phi)}$ be the toppling function of this process. Let $F$ be a function such that $\phi+\Delta F\geq 0, F(w)\geq 0$ for all $w$, and $F(v)\geq n$. Then $F(w)\geq H(w)$ for all $w$.
\end{corollary}
\begin{proof}
We apply Proposition~\ref{prop_waveleast} $n$ times, each time decreasing $F$ by $H^v_{W^k(\phi)}$ for $k=0,1,\dots, n-1$.
\end{proof}

\bibliography{../../sand/sandbib}
%\bibliography{solitonmaster}
\bibliographystyle{abbrv}

%\bibliographystyle{plain}
%\bibliography{main}

\end{document}